\documentclass[a4paper]{amsart}
\pdfoutput=1
\usepackage{etex}
\usepackage[T1]{fontenc}
\usepackage{lmodern}
\usepackage[utf8]{inputenc}
\usepackage{amssymb}
\usepackage{enumitem,booktabs}

\setlist[enumerate,1]{label=\textup{(\arabic*)}}

\usepackage{mathrsfs,dsfont,mathtools}
\usepackage{nicefrac}

\usepackage{tikz-cd}

\usepackage[pdftitle={Dagger completions and bornological torsion-freeness}, pdfauthor={Ralf Meyer, Devarshi Mukherjee},pdfsubject={Mathematics}]{hyperref}

\usepackage[lite]{amsrefs}

\renewcommand*{\PrintDOI}[1]{\href{http://dx.doi.org/\detokenize{#1}}{doi: \detokenize{#1}}}

\BibSpec{book}{%
  +{}  {\PrintPrimary}                {transition}
  +{,} { \textit}                     {title}
  +{.} { }                            {part}
  +{:} { \textit}                     {subtitle}
  +{,} { \PrintEdition}               {edition}
  +{}  { \PrintEditorsB}              {editor}
  +{,} { \PrintTranslatorsC}          {translator}
  +{,} { \PrintContributions}         {contribution}
  +{,} { }                            {series}
  +{,} { \voltext}                    {volume}
  +{,} { }                            {publisher}
  +{,} { }                            {organization}
  +{,} { }                            {address}
  +{,} { \PrintDateB}                 {date}
  +{,} { }                            {status}
  +{}  { \parenthesize}               {language}
  +{}  { \PrintTranslation}           {translation}
  +{;} { \PrintReprint}               {reprint}
  +{.} { }                            {note}
  +{.} {}                             {transition}
  +{} { \PrintDOI}                   {doi}
  +{} { available at \url}            {eprint}
  +{}  {\SentenceSpace \PrintReviews} {review}
}

\usepackage{microtype}

\numberwithin{equation}{section}

\theoremstyle{plain}
\newtheorem{theorem}[equation]{Theorem}
\newtheorem{lemma}[equation]{Lemma}

\newtheorem{proposition}[equation]{Proposition}
\newtheorem{corollary}[equation]{Corollary}
\theoremstyle{definition}
\newtheorem{definition}[equation]{Definition}

\theoremstyle{remark}

\newtheorem{example}[equation]{Example}

\newcommand{\dvr}{V}
\newcommand{\dvgen}{\pi}
\newcommand{\dvf}{K}
\newcommand{\resf}{k}

\newcommand*{\Z}{\mathbb Z}
\newcommand*{\N}{\mathbb N}

\newcommand*{\nb}{\nobreakdash}  

\newcommand{\bdd}{\mathcal B}
\newcommand{\coma}{\widehat}
\newcommand{\comb}{\overbracket[.7pt][1.4pt]}

\newcommand*{\ling}[1]{#1_\mathrm{lg}}
\newcommand*{\torf}[1]{#1_\mathrm{tf}}
\newcommand*{\comling}[1]{\comb{#1_\mathrm{lg}}}

\newcommand{\updagger}{{\textup{\tiny\!\! \textdagger}}}

\newcommand{\defeq}{\mathrel{:=}} 

\newcommand*{\into}{\rightarrowtail}

\newcommand*{\onto}{\twoheadrightarrow}

\mathtoolsset{mathic}
\DeclarePairedDelimiter{\floor}{\lfloor}{\rfloor}
\DeclarePairedDelimiter{\ceil}{\lceil}{\rceil}
\DeclarePairedDelimiter{\norm}{\lVert}{\rVert}
\DeclarePairedDelimiter{\abs}{\lvert}{\rvert}
\DeclarePairedDelimiterX{\setgiven}[2]{\{}{\}}{#1\,{:}\,\mathopen{}#2}

\DeclareMathOperator{\Endo}{End}

\newcommand*{\injto}{\hookrightarrow}

\begin{document}
\title[Dagger completions and bornological torsion-freeness]{Dagger completions and\\bornological torsion-freeness}

\author{Ralf Meyer}
\email{rmeyer2@uni-goettingen.de}

\author{Devarshi Mukherjee}
\email{devarshi.mukherjee@mathematik.uni-goettingen.de}

\address{Mathematisches Institut\\
  Georg-August Universit\"at G\"ottingen\\
  Bunsenstra\ss{}e 3--5\\
  37073 G\"ottingen\\
  Germany}

\begin{abstract}
  We define a dagger algebra as a bornological algebra over a discrete
  valuation ring with three properties that are typical of
  Monsky--Washnitzer algebras, namely, completeness, bornological
  torsion-freeness and a certain spectral radius condition.  We study
  inheritance properties of the three properties that define a dagger
  algebra.  We describe dagger completions of bornological
  algebras in general and compute some noncommutative examples.
\end{abstract}
\maketitle

\section{Introduction}

In~\cite{Monsky-Washnitzer:Formal}, Monsky and Washnitzer introduce a
cohomology theory for affine non-singular varieties defined over a
field~\(\resf\)
of nonzero characteristic.  Let~\(\dvr\)
be a discrete valuation ring such that the fraction field~\(\dvf\)
of~\(\dvr\)
has characteristic~\(0\).
Let \(\dvgen\in \dvr\)
be a uniformiser and let \(\resf=\dvr/\dvgen \dvr\)
be the residue field.  Monsky and Washnitzer lift the coordinate ring
of a smooth affine variety~\(X\)
over~\(\resf\)
to a smooth commutative algebra~\(A\)
over~\(\dvr\).
The dagger completion~\(A^\updagger\)
of~\(A\)
is a certain subalgebra of the \(\dvgen\)\nb-adic
completion of~\(A\).
If~\(A\)
is the polynomial algebra over~\(\dvr\),
then~\(A^\updagger\)
is the ring of overconvergent power series.  The Monsky--Washnitzer
cohomology is defined as the de Rham cohomology of the algebra
\(\dvf\otimes_\dvr A^\updagger\).

The dagger completion is interpreted
in~\cite{Cortinas-Cuntz-Meyer-Tamme:Nonarchimedean} in the setting of
bornological algebras, based on considerations about the joint
spectral radius of bounded subsets.  The main achievement
in~\cite{Cortinas-Cuntz-Meyer-Tamme:Nonarchimedean} is the
construction of a chain complex that computes the rigid cohomology of
the original variety~\(X\)
and that is strictly functorial.  In addition, this chain complex is
related to periodic cyclic homology.  Here we continue the study
of dagger completions.  We define dagger algebras by adding a
bornological torsion-freeness condition to the completeness and
spectral radius conditions already present
in~\cite{Cortinas-Cuntz-Meyer-Tamme:Nonarchimedean}.  We also show
that the category of dagger algebras is closed under extensions,
subalgebras, and certain quotients, by showing that all three
properties that define them are hereditary for these constructions.

The results in this article should help to reach the following
important goal: define an analytic cyclic cohomology theory for
algebras over~\(\resf\)
that specialises to Monsky--Washnitzer or rigid cohomology for the
coordinate rings of smooth affine varieties over~\(\resf\).
A general machinery for defining such cyclic cohomology theories is
developed in~\cite{Meyer:HLHA}.  It is based on a class of nilpotent
algebras, which must be closed under extensions.  This is why we are
particularly interested in properties hereditary for extensions.

If~\(S\)
is a bounded subset of a \(\dvf\)\nb-algebra~\(A\),
then its spectral radius \(\varrho(S)\in [0,\infty]\)
is defined in~\cite{Cortinas-Cuntz-Meyer-Tamme:Nonarchimedean}.
If~\(A\)
is a bornological \(\dvr\)\nb-algebra,
then only the inequalities \(\varrho(S) \le s\)
for \(s>1\)
make sense.  This suffices, however, to characterise the \emph{linear
  growth bornology} on a bornological \(\dvr\)\nb-algebra:
it is the smallest \(\dvr\)\nb-algebra
bornology with \(\varrho(S) \le 1\)
for all its bounded subsets~\(S\).
We call a bornological algebra~\(A\)
with this property \emph{semi-dagger} because this is the main feature
of dagger algebras.  Any bornological algebra~\(A\)
carries a smallest bornology with linear growth.  This defines a
semi-dagger algebra~\(\ling{A}\).
If~\(A\)
is a torsion-free, finitely generated, commutative \(\dvr\)\nb-algebra
with the fine bornology, then the bornological
completion~\(\comling{A}\)
of~\(\ling{A}\) is the Monsky--Washnitzer completion of~\(A\).

Any algebra over~\(\resf\)
is also an algebra over~\(\dvr\).
Equipped with the fine bornology, it is complete and semi-dagger.  We
prefer, however, not to call such algebras ``dagger algebras.''  The
feature of Monsky--Washnitzer algebras that they lack is
torsion-freeness.  The purely algebraic notion of torsion-freeness
does not work well for bornological algebras.  In particular, it is
unclear whether it is preserved by completions.  We call a
bornological \(\dvr\)\nb-module~\(A\)
\emph{bornologically torsion-free} if multiplication by~\(\dvgen\)
is a bornological isomorphism onto its image.  This notion has very
good formal properties: it is preserved by bornological completions
and linear growth bornologies and hereditary for subalgebras and
extensions.  So~\(\comling{A}\)
remains bornologically torsion-free if~\(A\)
is bornologically torsion-free.  The bornological version of
torsion-freeness coincides with the usual one for bornological
\(\dvr\)\nb-modules
with the fine bornology.  Thus~\(\comling{A}\)
is bornologically torsion-free if~\(A\)
is a torsion-free \(\dvr\)\nb-algebra with the fine bornology.

A bornological \(\dvr\)\nb-module~\(M\)
is bornologically torsion-free if and only if the canonical map
\(M\to \dvf\otimes_\dvr M\)
is a bornological embedding.  This property is very important.  On the
one hand, we must keep working with modules over~\(\dvr\)
in order to keep the original algebra over~\(\resf\)
in sight and because the linear growth bornology only makes sense for
algebras over~\(\dvr\).
On the other hand, we often need to pass to the \(\dvf\)\nb-vector
space \(\dvf\otimes_\dvr M\)
-- this is how de Rham cohomology is defined.  Bornological vector
spaces over~\(\dvf\)
have been used recently to do analytic geometry in
\cites{Bambozzi:Affinoid, Bambozzi-Ben-Bassat:Dagger,
  Bambozzi-Ben-Bassat-Kremnizer:Stein}.  The spectral radius of a
bounded subset of a bornological \(\dvr\)\nb-algebra~\(A\)
is defined in~\cite{Cortinas-Cuntz-Meyer-Tamme:Nonarchimedean} by
working in \(\dvf\otimes_\dvr A\),
which only works well if~\(A\)
is bornologically torsion-free.  Here we define a truncated spectral
radius in \([1,\infty]\)
without reference to \(\dvf\otimes_\dvr A\),
in order to define semi-dagger algebras independently of torsion
issues.

We prove that the properties of being complete, semi-dagger, or
bornologically torsion-free are hereditary for extensions.  Hence an
extension of dagger algebras is again a dagger algebra.

To illustrate our theory, we describe the dagger completions of monoid
algebras and crossed products.  Dagger completions of monoid algebras
are straightforward generalisations of Monsky--Washnitzer completions
of polynomial algebras.

We thank the anonymous referee for helpful comments to improve the
presentation in the paper.

\section{Basic notions}
\label{Review_bornologies}

In this section, we recall some basic notions on bornological modules
and bounded homomorphisms.
See~\cite{Cortinas-Cuntz-Meyer-Tamme:Nonarchimedean} for more details.
We also study the inheritance properties of separatedness and
completeness for submodules, quotients and extensions.

A \emph{bornology} on a set~\(X\)
is a collection~\(\bdd_X\)
of subsets of~\(X\),
called \emph{bounded} sets, such that all finite subsets are bounded
and subsets and finite unions of bounded subsets are bounded.
Let~\(\dvr\)
be a complete discrete valuation ring.  A \emph{bornological
  \(\dvr\)\nb-module}
is a \(\dvr\)\nb-module~\(M\)
with a bornology such that every bounded subset is contained in a
bounded \(\dvr\)\nb-submodule.
In particular, the \(\dvr\)\nb-submodule
generated by a bounded subset is again bounded.  We always
write~\(\bdd_M\) for the bornology on~\(M\).

Let \(M'\subseteq M\)
be a \(\dvr\)\nb-submodule.
The \emph{subspace bornology} on~\(M'\)
consists of all subsets of~\(M'\)
that are bounded in~\(M\).
The \emph{quotient bornology} on~\(M/M'\)
consists of all subsets of the form \(q(S)\)
with \(S\in\bdd_M\),
where \(q \colon M \to M/M'\)
is the canonical projection.  We always equip submodules and quotients
with these canonical bornologies.

Let~\(M\)
and~\(N\)
be two bornological \(\dvr\)\nb-modules.
A \(\dvr\)\nb-module
map \(f\colon M \to N\)
is \emph{bounded} if \(f(S)\in \bdd_N\)
for all \(S\in \bdd_M\).
Bornological \(\dvr\)\nb-modules
and bounded \(\dvr\)\nb-module
maps form an additive category.  The isomorphisms in this category are
called \emph{bornological isomorphisms}.  A bounded \(\dvr\)\nb-module
map \(f \colon M \to N\)
is a \emph{bornological embedding} if the induced map \(M \to f(M)\)
is a bornological isomorphism, where \(f(M)\subseteq N\)
carries the subspace bornology.  It is a \emph{bornological quotient
  map} if the induced map \(M/\ker f \to N\)
is a bornological isomorphism.  Equivalently, for each \(T\in\bdd_N\)
there is \(S\in\bdd_M\) with \(f(S) = T\).

An \emph{extension} of bornological \(\dvr\)\nb-modules
is a diagram of \(\dvr\)\nb-modules
\[
M' \xrightarrow{f} M \xrightarrow{g} M''
\]
that is algebraically exact and such that~\(f\)
is a bornological embedding and~\(g\)
a bornological quotient map.  Equivalently, \(g\)
is a cokernel of~\(f\)
and~\(f\)
a kernel of~\(g\)
in the additive category of bornological \(\dvr\)\nb-modules.
A \emph{split extension} is an extension with a bounded
\(\dvr\)\nb-linear
map \(s\colon M'' \to M\) such that \(g\circ s = \mathrm{id}_{M''}\).

Let~\(M\)
be a bornological \(\dvr\)\nb-module.
A sequence \((x_n)_{n\in\N}\)
in~\(M\)
\emph{converges} towards \(x\in M\)
if there are \(S\in\bdd_M\)
and a sequence \((\delta_n)_{n\in\N}\)
in~\(\dvr\)
with \(\lim {}\abs{\delta_n} = 0\)
and \(x_n-x \in \delta_n\cdot S\)
for all \(n\in\N\).
It is a \emph{Cauchy} sequence if there are \(S\in\bdd_M\)
and a sequence \((\delta_n)_{n\in\N}\)
in~\(\dvr\)
with \(\lim {}\abs{\delta_n} = 0\)
and \(x_n-x_m \in \delta_j\cdot S\)
for all \(n,m,j\in\N\)
with \(n,m\ge j\).
Since any bounded subset is contained in a bounded
\(\dvr\)\nb-submodule,
a sequence in~\(M\)
converges or is Cauchy if and only if it converges or is Cauchy in the
\(\dvgen\)\nb-adic
topology on some bounded \(\dvr\)\nb-submodule of~\(M\).

We call a subset~\(S\)
of~\(M\)
\emph{closed} if \(x\in S\)
for any sequence in~\(S\)
that converges in~\(M\)
to \(x\in M\).
These are the closed subsets of a topology on~\(M\).
Bounded maps preserve convergent sequences and Cauchy sequences.  Thus
pre-images of closed subsets under bounded maps remain closed.  That
is, bounded maps are continuous for these canonical topologies.

\subsection{Separated bornological modules}
\label{sec:separated}

We call~\(M\)
\emph{separated} if limits of convergent sequences in~\(M\)
are unique.  If~\(M\)
is not separated, then the constant sequence~\(0\)
has a non-zero limit.  Therefore, \(M\)
is separated if and only if \(\{0\}\subseteq M\)
is closed.  And~\(M\)
is separated if and only if any \(S\in\bdd_M\)
is contained in a \(\dvgen\)\nb-adically
separated bounded \(\dvr\)\nb-submodule.

\begin{lemma}
  \label{lem:separated_hereditary}
  Let \(M' \xrightarrow{f} M \xrightarrow{g} M''\)
  be an extension of bornological \(\dvr\)\nb-modules.
  \begin{enumerate}
  \item \label{lem:separated_hereditary_1}%
    If~\(M\) is separated, so is~\(M'\).
  \item \label{lem:separated_hereditary_2}%
    The quotient~\(M''\)
    is separated if and only if~\(f(M')\) is closed in~\(M\).
  \item \label{lem:separated_hereditary_3}%
    If \(M'\)
    and~\(M''\)
    are separated and~\(M''\)
    is torsion-free, then~\(M\) is separated.
  \end{enumerate}
\end{lemma}

\begin{proof}
  Assertion~\ref{lem:separated_hereditary_1} is trivial.

  If \(M''\)
  is separated, then \(\{0\}\subseteq M''\)
  is closed.  Hence \(g^{-1}(\{0\}) = f(M')\)
  is closed in~\(M\).
  If~\(M''\)
  is not separated, then the constant sequence~\(0\)
  in~\(M''\)
  converges to some non-zero \(x''\in M''\).
  That is, there are a bounded subset \(S''\subseteq M''\)
  and a null sequence~\((\delta_n)_{n\in\N}\)
  in~\(\dvr\)
  with \(x''-0 \in \delta_n\cdot S''\)
  for all \(n\in\N\).
  Since~\(g\)
  is a bornological quotient map, there are \(x\in M\)
  and \(S\in\bdd_M\)
  with \(g(x)=x''\)
  and \(g(S) = S''\).
  We may choose \(y''_n\in S''\)
  with \(x'' = \delta_n\cdot y''_n\)
  and \(y_n\in S\)
  with \(g(y_n) = y_n''\).
  So \(g(x - \delta_n y_n)=0\).
  Thus the sequence \((x-\delta_n y_n)\)
  lies in~\(f(M')\).
  It converges to~\(x\),
  which does not belong to~\(f(M')\)
  because \(x''\neq0\).
  So~\(f(M')\)
  is not closed.  This finishes the proof
  of~\ref{lem:separated_hereditary_2}.

  We prove~\ref{lem:separated_hereditary_3}.  Let \(x\in M\)
  belong to the closure of~\(\{0\}\)
  in~\(M\).
  That is, there are \(S\in\bdd_M\)
  and a null sequence \((\delta_n)_{n\in\N}\)
  in~\(\dvr\)
  with \(x\in \delta_n \cdot S\)
  for all \(n\in\N\).
  Then \(g(x) \in \delta_n \cdot g(S)\)
  for all \(n\in\N\).
  This implies \(g(x)=0\)
  because~\(M''\)
  is separated.  So there is \(y\in M'\)
  with \(f(y)=x\).
  And \(f(y) = x \in \delta_n\cdot S\).
  Choose \(x_n\in S\)
  with \(f(y) = \delta_n\cdot x_n\).
  We may assume \(\delta_n\neq0\)
  for all \(n\in\N\)
  because otherwise \(x\in\delta_n\cdot S\)
  is~\(0\).
  Since~\(M''\)
  is torsion-free, \(\delta_n\cdot x_n\in f(M')\)
  implies \(g(x_n)=0\).
  So we may write \(x_n=f(y_n)\)
  for some \(y_n\in M'\).
  Since~\(f\)
  is a bornological embedding, the set \(\setgiven{y_n}{n\in\N}\)
  in~\(M'\)
  is bounded.  Since~\(M'\)
  is separated and \(y = \delta_n\cdot y_n\)
  for all \(n\in\N\),
  we get \(y=0\).  Hence \(x=0\).  So~\(\{0\}\) is closed in~\(M\).
\end{proof}

The quotient \(M/\overline{\{0\}}\)
of a bornological \(\dvr\)\nb-module~\(M\)
by the closure of~\(0\)
is called the \emph{separated quotient} of~\(M\).
It is separated by Lemma~\ref{lem:separated_hereditary}, and it is the
largest separated quotient of~\(M\).
Even more, the quotient map \(M\to M/\overline{\{0\}}\)
is the universal arrow to a separated bornological \(\dvr\)\nb-module,
that is, any bounded \(\dvr\)\nb-linear
map from~\(M\)
to a separated bornological \(\dvr\)\nb-module
factors uniquely through \(M/\overline{\{0\}}\).

The following example shows that
Lemma~\ref{lem:separated_hereditary}.\ref{lem:separated_hereditary_3}
fails without the torsion-freeness assumption.

\begin{example}
  \label{exa:extension_not_separated}
  Let \(M'=\dvr\)
  and let \(M = \dvr[x]/ S\),
  where~\(S\)
  is the \(\dvr\)\nb-submodule
  of~\(\dvr[x]\)
  generated by \(1 - \dvgen^n x^n\)
  for all \(n\in \N\).
  We embed \(M'=\dvr\) as multiples of \(1=x^0\).  Then
  \[
  M/M' = \bigoplus_{n=1}^\infty \dvr/(\dvgen^n),
  \]
  We endow \(M\),
  \(M'\)
  and \(M/M'\)
  with the bornologies where all subsets are bounded.  We get an
  extension of bornological \(\dvr\)\nb-modules
  \(\dvr \into M \onto \bigoplus_{n=1}^{\infty}\dvr/ (\dvgen^n)\).
  Here \(\dvr\)
  and \(\bigoplus_{n=1}^\infty \dvr/(\dvgen^n)\)
  are \(\dvgen\)\nb-adically
  separated, but~\(M\)
  is not: the constant sequence~\(1\)
  in~\(M\)
  converges to~\(0\)
  because \(1 = 1-\dvgen^n x^n + \dvgen^n x^n \equiv \dvgen^n x^n\)
  in~\(M\).
\end{example}

\subsection{Completeness}
\label{sec:complete}

We call a bornological \(\dvr\)\nb-module~\(M\)
\emph{complete} if it is separated and for any \(S\in\bdd_M\)
there is \(T\in\bdd_M\)
so that all \(S\)\nb-Cauchy
sequences are \(T\)\nb-convergent.
Equivalently, any \(S\in\bdd_M\)
is contained in a \(\dvgen\)\nb-adically
complete bounded \(\dvr\)\nb-submodule
(see
\cite{Cortinas-Cuntz-Meyer-Tamme:Nonarchimedean}*{Proposition~2.8}).
By definition, any Cauchy sequence in a complete bornological
\(\dvr\)\nb-module has a unique limit.

\begin{theorem}
  \label{the:extension_complete}
  Let \(M' \xrightarrow{f} M \xrightarrow{g} M''\)
  be an extension of bornological \(\dvr\)\nb-modules.
  \begin{enumerate}
  \item \label{the:extension_complete_1}%
    If~\(M\)
    is complete and~\(f(M')\)
    is closed in~\(M\), then~\(M'\) is complete.
  \item \label{the:extension_complete_1b}%
    If~\(M'\)
    is complete, \(M\)
    separated, and~\(M''\)
    torsion-free, then~\(f(M')\) is closed in~\(M\).
  \item \label{the:extension_complete_2}%
    Let~\(M\)
    be complete.  Then~\(M''\)
    is complete if and only if~\(f(M')\) is closed in~\(M\).
  \item \label{the:extension_complete_3}%
    If \(M'\)
    and~\(M''\)
    are complete and~\(M\) is separated, then~\(M\) is complete.
    If \(M'\)
    and~\(M''\)
    are complete and~\(M''\) is torsion-free, then~\(M\) is complete.
  \end{enumerate}
\end{theorem}

\begin{proof}
  Statement~\ref{the:extension_complete_1}
  is~\cite{Cortinas-Cuntz-Meyer-Tamme:Nonarchimedean}*{Lemma~2.13},
  and there is no need to repeat the proof here.  It is somewhat
  similar to the proof of~\ref{the:extension_complete_3}.  Next we
  prove~\ref{the:extension_complete_1b}.  Assume that~\(M'\)
  is complete, that \(M''\)
  is torsion-free, and that~\(f(M')\)
  is not closed in~\(M\).
  We are going to prove that~\(M\)
  is not separated.  There is a sequence \((x_n)_{n\in\N}\)
  in~\(M'\)
  for which \(f(x_n)_{n\in\N}\)
  converges in~\(M\)
  towards some \(x\notin f(M')\).
  So there is a bounded set \(S\subseteq M\)
  and a sequence \((\delta_k)_{k\in\N}\)
  in~\(\dvr\)
  with \(\lim {}\abs{\delta_k} = 0\)
  and \(f(x_n)-x \in \delta_n\cdot S\)
  for all \(n\in\N\).
  We may assume without loss of generality that~\(S\)
  is a bounded \(\dvr\)\nb-submodule
  and that the sequence of norms~\(\abs{\delta_n}\)
  is decreasing: let~\(\delta_n^*\)
  be the~\(\delta_m\)
  for \(m\ge n\)
  with maximal norm.  Then
  \(f(x_n) - x \in \delta_n\cdot S\subseteq \delta_n^* \cdot S\)
  and still \(\lim {}\abs{\delta_n^*} =0\).
  We may write \(f(x_n) - x = \delta^*_n y_n\)
  with \(y_n \in S\).
  Let \(m<n\).
  Then \(\delta^*_m g(y_m) = -g(x) = \delta_n^* g(y_n)\)
  and \(\delta_n^*/\delta_m^* \in \dvr\);
  this implies first
  \(\delta^*_m\cdot \bigl(g(y_m)- g(y_n)
  \cdot\delta_n^*/\delta_m^*\bigr) = 0\)
  and then \(g(y_m)= g(y_n) \cdot \delta_n^*/\delta_m^*\)
  because~\(M''\)
  is torsion-free.  So there is \(z_{m,n} \in M'\)
  with \(y_m + f(z_{m,n}) = y_n\cdot \delta_n^*/\delta_m^*\).
  We even have \(z_{m,n} \in f^{-1}(S)\)
  because~\(S\)
  is a \(\dvr\)\nb-submodule.
  The subset \(f^{-1}(S)\subseteq M'\)
  is bounded because~\(f\)
  is a bornological embedding.  We get
  \(f(x_n) - f(x_m) = \delta^*_n y_n - \delta^*_m y_m = f(\delta_m^*
  z_{m,n})\)
  and hence \(x_n - x_m = \delta_m^* z_{m,n}\)
  for \(n>m\).
  This witnesses that the sequence~\((x_n)_{n\in\N}\)
  is Cauchy in~\(M'\).
  Since~\(M'\)
  is complete, it converges towards some \(y\in M'\).
  Then \(f(x_n)\)
  converges both towards \(f(y)\in f(M')\)
  and towards \(x\notin f(M')\).
  So~\(M\)
  is not separated.  This finishes the proof
  of~\ref{the:extension_complete_1b}.

  Next we prove~\ref{the:extension_complete_2}.  If~\(f(M')\)
  is not closed, then Lemma~\ref{lem:separated_hereditary} shows
  that~\(M''\)
  is not separated and hence not complete.  Conversely, we claim
  that~\(M''\)
  is complete if~\(f(M')\)
  is closed.  Lemma~\ref{lem:separated_hereditary} shows that~\(M''\)
  is separated.  Let \(S''\in\bdd_{M''}\).
  There is \(S\in\bdd_M\)
  with \(g(S)=S''\)
  because~\(g\)
  is a bornological quotient map.  And there is \(T\in\bdd_M\)
  so that any \(S\)\nb-Cauchy
  sequence is \(T\)\nb-convergent.
  We claim that any \(S''\)\nb-Cauchy
  sequence is \(g(T)\)\nb-convergent.
  So let \((x''_n)_{n\in\N}\)
  be an \(S''\)\nb-Cauchy
  sequence.  Thus there is a null sequence \((\delta_n)_{n\in\N}\)
  in~\(\dvr\)
  with \(x_n'' - x''_m \in \delta_j \cdot S''\)
  for all \(n,m,j\in\N\)
  with \(n,m \ge j\).
  As above, we may assume without loss of generality that the sequence
  of norms~\(\abs{\delta_n}\)
  is decreasing.  Choose any \(x_0\in M\)
  with \(g(x_0) = x_0''\).
  For each \(n\in\N\),
  choose \(y_n\in S\)
  with \(x_{n+1}''-x''_n = \delta_n\cdot g(y_n)\).  Let
  \[
  x_n \defeq x_0 + \delta_0\cdot y_0 + \dotsb + \delta_{n-1}\cdot y_{n-1}.
  \]
  Then \(g(x_n) = x_n''\).
  And \(x_{n+1} - x_n = \delta_n\cdot y_n \in \delta_n \cdot S\).
  Since~\(\abs{\delta_n}\)
  is decreasing, this implies \(x_m - x_n \in \delta_n\cdot S\)
  for all \(m\ge n\).
  So the sequence~\((x_n)_{n\in\N}\)
  is \(S\)\nb-Cauchy.
  Hence it is \(T\)\nb-convergent.
  Thus \(g(x_n) = x_n''\)
  is \(g(T)\)\nb-convergent
  as asserted.  This finishes the proof
  of~\ref{the:extension_complete_2}.

  Finally, we prove~\ref{the:extension_complete_3}.  So we assume
  \(M'\)
  and~\(M''\)
  to be complete.  If~\(M''\)
  is torsion-free, then~\(M\)
  is separated by Lemma~\ref{lem:separated_hereditary}.  Hence the
  second statement in~\ref{the:extension_complete_3} is a special case
  of the first one.  Let \(S\in\bdd_M\).
  We must find \(T\in\bdd_M\)
  so that every \(S\)\nb-Cauchy
  sequence is \(T\)\nb-convergent.
  Since~\(M\)
  is separated, this says that it is complete.  Since~\(M''\)
  is complete, there is a \(\dvgen\)\nb-adically
  complete \(\dvr\)\nb-submodule
  \(T_0\in\bdd_{M''}\)
  that contains \(g(S)\).
  Since~\(g\)
  is a bornological quotient map, there is \(T_1\in\bdd_M\)
  with \(g(T_1) = T_0\).
  Replacing it by \(T_1+S\),
  we may arrange, in addition, that \(S\subseteq T_1\).
  Since \(f\)
  is a bornological embedding, \(T_2 \defeq f^{-1}(T_1)\)
  is bounded in~\(M'\).
  As~\(M'\)
  is complete, there is \(T_3\in\bdd_{M'}\)
  so that every \(T_2\)\nb-Cauchy
  sequence is \(T_3\)\nb-convergent.
  We claim that any \(S\)\nb-Cauchy
  sequence is \(T_1 + f(T_3)\)-convergent.
  The proof of this claim will finish the proof of the theorem.

  Let \((x_n)_{n\in\N}\)
  be an \(S\)\nb-Cauchy
  sequence.  So there are \(\delta_n\in \dvr\)
  and \(y_n\in S\)
  with \(\lim {}\abs{\delta_n} = 0\)
  and \(x_{n+1} - x_n = \delta_n \cdot y_n\).
  As above, we may assume that~\(\abs{\delta_n}\)
  is decreasing and that \(\delta_0=1\).
  Since \(g(y_{n+k})\in g(S) \subseteq T_0\)
  and~\(T_0\)
  is \(\dvgen\)\nb-adically
  complete, the following series converges in~\(T_0\):
  \begin{equation}
    \label{eq:def_tilde_wn}
    \tilde{w}_n \defeq
    -\sum_{k=0}^\infty \frac{\delta_{n+k}}{\delta_n} g(y_{n+k}).
  \end{equation}
  Since \(\tilde{w}_n\in T_0\),
  there is \(w_n\in T_1\) with \(g(w_n) = \tilde{w}_n\).  So
  \begin{multline*}
    \delta_n g(w_n)
    = \lim_{N\to\infty} -g\left(\sum_{k=0}^N \delta_{n+k} y_{n+k} \right)
    \\= \lim_{N\to\infty} g(x_n - x_{N+n+1})
    = g(x_n) - \lim_{N\to\infty} g(x_N).
  \end{multline*}
  In particular, \(g(w_0) = g(x_0) - \lim_{N\to\infty} g(x_N)\).
  Now let
  \[
  \tilde{x}_k \defeq x_k - \delta_k w_k + w_0 - x_0.
  \]
  Then
  \[
  g(\tilde{x}_k)
  = g(x_k) - g(x_k) + \lim_{N \to \infty} g(x_N)
  + g(x_0) - \lim_{N \to \infty} g(x_N) - g(x_0) = 0.
  \]
  So \(\tilde{x}_k\in f(M')\) for all \(k\in\N\).  And
  \begin{multline}
    \label{eq:tilde_x_Cauchy}
    \tilde{x}_{n+1} - \tilde{x}_n
    = x_{n+1} - x_n - \delta_{n+1}w_{n+1} + \delta_n w_n
    \\= \delta_n y_n + \delta_n w_n - \delta_{n+1}w_{n+1}
    = \delta_n \cdot \left(y_n + w_n -
    \frac{\delta_{n+1}}{\delta_n}w_{n+1}\right).
  \end{multline}
  Let \(z_n \defeq y_n + w_n -
  \frac{\delta_{n+1}}{\delta_n}w_{n+1}\).
  A telescoping sum argument shows that
  \begin{equation}
    \label{telescoping_sum}
    g(z_n)
    = g(y_n) + \tilde{w}_n
    - \frac{\delta_{n+1}}{\delta_n} \tilde{w}_{n+1}
    = 0.
  \end{equation}
  So \(z_n \in f(M')\).
  And \(z_n \in S + T_1 + T_1 = T_1\).
  Thus there is \(\hat{z}_n \in f^{-1}(T_1) = T_2\)
  with \(z_n = f(\hat{z}_n)\).
  Equation~\eqref{eq:tilde_x_Cauchy} means that the sequence
  \(f^{-1}(\tilde{x}_n)\)
  is \(T_2\)\nb-Cauchy.
  Hence it is \(T_3\)\nb-convergent.
  So~\((\tilde{x}_n)\)
  is \(f(T_3)\)\nb-convergent.
  Then~\((x_n)\) is \(T_1+f(T_3)\)-convergent.
\end{proof}

The following examples show that the technical extra assumptions in
\ref{the:extension_complete_1b} and~\ref{the:extension_complete_3} in
Theorem~\ref{the:extension_complete} are necessary.  They only involve
extensions of \(\dvr\)\nb-modules
with the bornology where all subsets are bounded.  For this bornology,
bornological completeness and separatedness are the same as
\(\dvgen\)\nb-adic
completeness and separatedness, respectively, and any extension of
\(\dvr\)\nb-modules is a bornological extension.

\begin{example}
  Let \(M' \defeq \{0\}\)
  and \(M \defeq \dvf\)
  with the bornology of all subsets.  Then~\(M'\)
  is bornologically complete, but not closed in~\(M\),
  and \(M/M' = M\)
  is torsion-free.  So
  Theorem~\ref{the:extension_complete}.\ref{the:extension_complete_1b}
  needs the assumption that~\(M\) be separated.
\end{example}

\begin{example}
  \label{exa:quotient_complete_by_non-closed}
  Let~\(M\)
  be the \(\dvr\)\nb-module
  of all power series \(\sum_{n=0}^\infty c_n x^n\)
  with \(\lim {}\abs{c_n} = 0\)
  and with the bornology where all subsets are bounded; this is the
  \(\dvgen\)\nb-adic
  completion of the polynomial algebra~\(\dvr[x]\).
  Let \(M' = M\)
  and define \(f\colon M' \to M\),
  \(f\bigl( \sum_{n=0}^\infty c_n x^n\bigr) \defeq \sum_{n=0}^\infty
  c_n \dvgen^n x^n\).
  This is a bornological embedding simply because all subsets in
  \(M=M'\)
  are bounded.  Let \(p_n \defeq \sum_{j=0}^n x^j\).
  This sequence in \(M'=M\)
  does not converge.  Nevertheless, the sequence
  \(f(p_n) = \sum_{j=0}^n \dvgen^j x^j\)
  converges in~\(M\)
  to \(\sum_{j=0}^\infty \dvgen^j x^j\).
  Thus \(f(M')\)
  is not closed in~\(M\),
  although \(M\)
  and~\(M'\)
  are complete and~\(f\)
  is a bornological embedding.  So
  Theorem~\ref{the:extension_complete}.\ref{the:extension_complete_1b}
  needs the assumption that~\(M''\) be torsion-free.
\end{example}

\begin{example}
  \label{exa:extension_of_complete_not_separated}
  We modify Example~\ref{exa:extension_not_separated} to produce an
  extension of \(\dvr\)\nb-modules
  \(N' \into N \onto N''\)
  where \(N'\)
  and~\(N''\)
  are \(\dvgen\)\nb-adically
  complete, but~\(N\)
  is not \(\dvgen\)\nb-adically
  separated and hence not \(\dvgen\)\nb-adically
  complete.  We let \(N' \defeq \dvr/(\dvgen) = \resf\).
  We let~\(N''\)
  be the \(\dvgen\)\nb-adic
  completion of the \(\dvr\)\nb-module~\(M''\)
  of Example~\ref{exa:extension_not_separated}.  That is,
  \[
  N'' \defeq \setgiven[\bigg]{ (c_n)_{n\in\N} \in \prod_{n=0}^\infty \dvr/(\dvgen^n)}{\lim{} \abs{c_n}=0}.
  \]
  This is indeed \(\dvgen\)\nb-adically complete.  So is
  \[
  N_1 \defeq
  \setgiven[\bigg]{ (c_n)_{n\in\N} \in \prod_{n=0}^\infty \dvr/(\dvgen^{n+1})}{\lim{} \abs{c_n}=0}.
  \]
  The kernel of the quotient map \(q\colon N_1\onto N''\)
  is isomorphic to
  \(\prod_{n=0}^\infty \dvr/(\dvgen) = \prod_\N \resf\).
  This is a \(\resf\)\nb-vector
  space, and it contains the \(\resf\)\nb-vector
  space \(\sum_{n=0}^\infty \resf\).
  Since any \(\resf\)\nb-vector
  space has a basis, we may extend the linear functional
  \(\sum_{n=0}^\infty \resf \to \resf\),
  \((c_n)_{n\in\N} \mapsto \sum_{n=0}^\infty c_n\),
  to a \(\resf\)\nb-linear
  functional \(\sigma\colon \prod_\N \resf \to \resf\).
  Let \(L\defeq \ker \sigma \subseteq \ker q\)
  and let \(N\defeq N_1/L\).
  The map~\(q\)
  descends to a surjective \(\dvgen\)\nb-linear
  map \(N \onto N''\).
  Its kernel is isomorphic to
  \(\prod_\N \resf/\ker \sigma \cong \resf = N'\).
  The functional \(\sigma\colon \prod_\N \resf \to \resf\)
  vanishes on \(\delta_0 - \delta_k\)
  for all \(k\in\N\),
  but not on~\(\delta_0\).
  When we identify \(\prod_\N \resf \cong \ker q\),
  we map~\(\delta_k\)
  to \(\dvgen^k \delta_k \in N_1\).
  So \(\delta_0\)
  and~\(\dvgen^k \delta_k\)
  get identified in~\(N\),
  but~\(\delta_0\)
  does not become~\(0\):
  it is the generator of \(N' = \dvr/(\dvgen)\)
  inside~\(N\).
  Since \([\delta_0] = \dvgen^k [\delta_k]\)
  in~\(N\),
  the \(\dvr\)\nb-module~\(N\)
  is not \(\dvgen\)\nb-adically separated.
\end{example}

The \emph{completion} \(\comb{M}\)
of a bornological \(\dvr\)\nb-module~\(M\)
is a complete bornological \(\dvr\)\nb-module
with a bounded \(\dvr\)\nb-linear
map \(M\to \comb{M}\)
that is universal in the sense that any bounded \(\dvr\)\nb-linear
map from~\(M\)
to a complete bornological \(\dvr\)\nb-module~\(X\)
factors uniquely through~\(\comb{M}\).
Such a completion exists and is unique up to isomorphism
(see~\cite{Cortinas-Cuntz-Meyer-Tamme:Nonarchimedean}*{Proposition~2.15}).
We shall describe it more concretely later when we need the details of
its construction.

\subsection{Vector spaces over the fraction field}
\label{sec:dvf-vector_spaces}

Recall that~\(\dvf\)
denotes the quotient field of~\(\dvr\).
Any \(\dvr\)\nb-linear
map between two \(\dvf\)\nb-vector
spaces is also \(\dvf\)\nb-linear.
So \(\dvf\)\nb-vector
spaces with \(\dvf\)\nb-linear
maps form a full subcategory in the category of \(\dvr\)\nb-modules.
A \(\dvr\)\nb-module~\(M\)
comes from a \(\dvf\)\nb-vector space if and only if the map
\begin{equation}
  \label{eq:mult_dvgen}
  \dvgen_M\colon M\to M,\qquad m\mapsto \dvgen\cdot m,
\end{equation}
is invertible.  We could define bornological \(\dvf\)\nb-vector
spaces without reference to~\(\dvr\).
Instead, we realise them as bornological \(\dvr\)\nb-modules
with an extra property:

\begin{definition}
  A bornological \(\dvr\)\nb-module~\(M\)
  is a \emph{bornological \(\dvf\)\nb-vector
    space} if the map~\(\dvgen_M\)
  in~\eqref{eq:mult_dvgen} is a bornological isomorphism, that is, an
  invertible map with bounded inverse.
\end{definition}

Given a bornological \(\dvr\)\nb-module~\(M\),
the tensor product \(\dvf\otimes M \defeq \dvf\otimes_\dvr M\)
with the tensor product bornology (see
\cite{Cortinas-Cuntz-Meyer-Tamme:Nonarchimedean}*{Lemma~2.18}) is a
bornological \(\dvf\)\nb-vector
space because multiplication by~\(\dvgen\)
is a bornological isomorphism on~\(\dvf\).

\begin{lemma}
  \label{lem:universal_to_vector_space}
  The canonical bounded \(\dvr\)\nb-linear map
  \[
  \iota_M \colon M \to \dvf\otimes M,\qquad m \mapsto 1 \otimes m,
  \]
  is the universal arrow from~\(M\)
  to a bornological \(\dvf\)\nb-vector
  space, that is, any bounded \(\dvr\)\nb-linear
  map \(f\colon M\to N\)
  to a bornological \(\dvf\)\nb-vector
  space~\(N\)
  factors uniquely through a bounded \(\dvr\)\nb-linear
  map \(f^\#\colon \dvf\otimes M \to N\),
  and this map is also \(\dvf\)\nb-linear.
\end{lemma}

\begin{proof}
  A \(\dvr\)\nb-linear
  map \(f^\#\colon \dvf\otimes M \to N\)
  must be \(\dvf\)\nb-linear.
  Hence the only possible candidate is the \(\dvf\)\nb-linear
  map defined by \(f^\#(x\otimes m) \defeq x\cdot f(m)\)
  for \(m\in M\),
  \(x\in \dvf\).
  Any bounded submodule of \(\dvf\otimes M\)
  is contained in \(\dvgen^{-k}\dvr \otimes S\)
  for some bounded submodule \(S\subseteq M\)
  and some \(k\in\N\),
  and \(f^\#(\dvgen^{-k}\dvr \otimes S) = \dvgen_N^{-k}(f(S))\)
  is bounded in~\(N\)
  because~\(\dvgen_N\)
  is a bornological isomorphism.  Thus~\(f^\#\) is bounded.
\end{proof}

\section{Spectral radius and semi-dagger algebras}
\label{sec:semi-dagger}

A \emph{bornological \(\dvr\)\nb-algebra}
is a bornological \(\dvr\)\nb-module~\(A\)
with a bounded, \(\dvr\)\nb-linear,
associative multiplication.  We do not assume~\(A\)
to have a unit element.  We fix a bornological
\(\dvr\)\nb-algebra~\(A\) throughout this section.

We recall some definitions
from~\cite{Cortinas-Cuntz-Meyer-Tamme:Nonarchimedean}.  Let
\(\varepsilon = \abs{\dvgen}\).
Let \(S\in\bdd_A\)
and let \(r\le1\).
There is a smallest integer~\(j\)
with \(\varepsilon^j \le r\),
namely, \(\ceil{\log_{\varepsilon}(r)}\).  Define
\[
r \star S \defeq \dvgen^{\ceil{\log_{\varepsilon}(r)}} \cdot S.
\]
Let \(\sum_{n=1}^\infty r^n \star S^n\)
be the \(\dvr\)\nb-submodule
generated by \(\bigcup_{n=1}^\infty r^n \star S^n\).
That is, its elements are finite \(\dvr\)\nb-linear
combinations of elements in \(\bigcup_{n=1}^\infty r^n \star S^n\).

\begin{definition}
  The \emph{truncated spectral radius}
  \(\varrho_1(S) = \varrho_1(S;\bdd_A)\)
  of \(S \in\bdd_A\)
  is the infimum of all \(r\ge 1\)
  for which \(\sum_{n=1}^\infty r^{-n} \star S^n\)
  is bounded.  It is~\(\infty\) if no such~\(r\) exists.
\end{definition}

By definition, \(\varrho_1(S) \in [1,\infty]\).
If~\(A\)
is an algebra over the fraction field~\(\dvf\)
of~\(\dvr\),
then we may define \(\sum_{n=1}^\infty r^{-n} \star S^n\)
also for \(0<r<1\).
Then the full spectral radius \(\varrho(S) \in [0,\infty]\)
is defined like~\(\varrho_1(S)\),
but without the restriction to \(r\ge1\).
The arguments in the beginning of Section~3.1
in~\cite{Cortinas-Cuntz-Meyer-Tamme:Nonarchimedean} assume implicitly
that the \(\dvr\)\nb-algebra~\(A\)
is a bornological subalgebra in a \(\dvf\)\nb-algebra,
so that the spectral radius is defined without truncation.  This means
that~\(A\)
is bornologically torsion-free (see
Proposition~\ref{pro:bornological_torsion_characterisation}).  The
results in
\cite{Cortinas-Cuntz-Meyer-Tamme:Nonarchimedean}*{Section~3.1} work
without this assumption if the truncated spectral radius is used
througout and the following lemma is used instead of
\cite{Cortinas-Cuntz-Meyer-Tamme:Nonarchimedean}*{Lemma~3.1.2}:

\begin{lemma}
  \label{lgb_equivalence}
  Let \(S\subseteq A\)
  be a bounded \(\dvr\)\nb-submodule
  and \(m\in \N_{\ge1}\).
  Then \(\varrho_1(S) = 1\)
  if and only if \(\sum_{l=1}^\infty (\dvgen^m S^j)^l\)
  is bounded for all \(j \in \N_{\ge1}\).
\end{lemma}

\begin{proof}
  Let \(\varrho_1(S)= 1\)
  and \(j \in \N_{\ge1}\).
  Then
  \(\sum_{l=1}^\infty \dvgen^{m \cdot l} S^{j\cdot l} \subseteq
  \sum_{k=1}^\infty (\varepsilon^{m/j})^k \star S^k\)
  is bounded because \(\varepsilon^{m/j}<1\).
  Conversely, let \(\sum_{l=1}^\infty (\dvgen^m S^j)^l\)
  be bounded.  Then \(\varrho_1(\dvgen^m S^j)= 1\).
  The proof of
  \cite{Cortinas-Cuntz-Meyer-Tamme:Nonarchimedean}*{Lemma~3.1.2} shows
  that \(\varrho_1(S) \le \varepsilon^{-m/j}\).
  This inequality for all \(j\in \N_{\ge1}\)
  implies \(\varrho_1(S) = 1\).
\end{proof}

\begin{definition}
  A bornological \(\dvr\)\nb-algebra~\(A\)
  is \emph{semi-dagger} if \(\varrho_1(S) = 1\)
  for all \(S\in\bdd_A\).
\end{definition}

\begin{proposition}[\cite{Cortinas-Cuntz-Meyer-Tamme:Nonarchimedean}*{Proposition~3.1.3}]
  \label{pro:semi-dagger_criterion}
  A bornological \(\dvr\)\nb-algebra~\(A\)
  is semi-dagger if and only if \(\sum_{i=0}^\infty \dvgen^i S^{i+1}\)
  is bounded for all \(S\in\bdd_A\),
  if and only if \(\sum_{i=0}^\infty \dvgen^i S^{ci+d}\)
  is bounded for all \(S\in\bdd_A\)
  and \(c,d\in \N\)
  with \(d\ge 1\),
  if and only if any \(S\in\bdd_A\)
  is contained in a bounded \(\dvr\)\nb-submodule
  \(U\subseteq A\) with \(\dvgen\cdot U\cdot U \subseteq U\).
\end{proposition}

\begin{definition}
  The \emph{linear growth bornology} on a bornological
  \(\dvr\)\nb-algebra~\(A\)
  is the smallest semi-dagger bornology on~\(A\).
  That is, it is the smallest bornology~\(\bdd'_A\)
  with \(\varrho_1(S;\bdd'_A)= 1\)
  for all \(S\in\bdd'_A\).
  Let~\(\ling{A}\) be~\(A\) with the linear growth bornology.
\end{definition}

The existence of a smallest semi-dagger bornology is shown
in~\cite{Cortinas-Cuntz-Meyer-Tamme:Nonarchimedean} by describing it
explicitly as follows:

\begin{lemma}[\cite{Cortinas-Cuntz-Meyer-Tamme:Nonarchimedean}*{Proposition~3.1.3
    and Lemma~3.1.10}]
  \label{equivalent_semidagger}
  Let \(T\subseteq A\).  The following are equivalent:
  \begin{enumerate}
  \item \label{equivalent_semidagger_1}%
    \(T\) is bounded in~\(\ling{A}\);
  \item \label{equivalent_semidagger_2}%
    \(T \subseteq \sum_{i=0}^\infty \dvgen^i S^{i+1}\)
    for some \(S\in\bdd_A\).
  \item \label{equivalent_semidagger_3}%
    \(T \subseteq \sum_{i=0}^\infty \dvgen^i S^{ci+d}\)
    for some \(S\in\bdd_A\) and \(c,d\in \N\) with \(d\ge 1\).
  \end{enumerate}
\end{lemma}

More precisely, the proof of
\cite{Cortinas-Cuntz-Meyer-Tamme:Nonarchimedean}*{Lemma~3.1.10} shows
that the subsets in~\ref{equivalent_semidagger_2} form a semi-dagger
bornology that contains~\(\bdd_A\).
And Proposition~\ref{pro:semi-dagger_criterion} shows that these
subsets are bounded in any semi-dagger bornology on~\(A\)
that contains~\(\bdd_A\).
So they form the smallest semi-dagger bornology containing~\(\bdd_A\).

By definition, the algebra~\(\ling{A}\)
has the following universal property: if~\(B\)
is a semi-dagger \(\dvr\)\nb-algebra,
then an algebra homomorphism \(A \to B\)
is bounded if and only if it is bounded on~\(\ling{A}\).
The algebra~\(A\)
is semi-dagger if and only if \(A=\ling{A}\).

\begin{theorem}
  \label{the:extension_lgb}
  Let \(A \xrightarrow{i} B \xrightarrow{q} C\)
  be an extension of bornological \(\dvr\)\nb-algebras.
  Then~\(B\)
  is a semi-dagger algebra if and only if both \(A\)
  and~\(C\) are.
\end{theorem}

\begin{proof}
  First assume~\(B\)
  to be semi-dagger.  Let \(S\in\bdd_A\).
  Then \(\sum_{j=0}^\infty \dvgen^j i(S)^{j+1}\)
  is bounded in~\(B\).
  Since~\(i\)
  is a bornological embedding, it follows that
  \(\sum_{j=0}^\infty \dvgen^j S^{j+1}\)
  is bounded in~\(A\).
  That is, \(\varrho_1(S;\bdd_A) = 1\).
  So~\(A\)
  is semi-dagger.  Now let \(S\in\bdd_C\).
  Since~\(q\)
  is a bornological quotient map, there is \(T\in\bdd_B\)
  with \(q(T) = S\).
  The subset \(\sum_{j=0}^\infty \dvgen^j T^{j+1}\)
  is bounded in~\(B\)
  because~\(B\)
  is semi-dagger.  Its image under~\(q\)
  is also bounded, and this is \(\sum_{j=0}^\infty \dvgen^j S^{j+1}\).
  So \(\varrho_1(S;\bdd_C) = 1\) and~\(C\) is semi-dagger.

  Now assume that \(A\)
  and~\(C\)
  are semi-dagger.
  We show that \(\sum_{l=1}^\infty(\dvgen^2 S^j)^l\)
  is bounded in~\(B\)
  for all \(S\in\bdd_B\),
  \(j\in\N_{\ge1}\).
  This implies \(\varrho_1(S;\bdd_B)=1\)
  by Lemma~\ref{lgb_equivalence}.

  Since~\(C\)
  is semi-dagger, \(\varrho_1(q(S);\bdd_C)= 1\).
  Thus \(S_2 \defeq \sum_{l=1}^\infty q(\dvgen S^j)^l\)
  is bounded in~\(C\)
  by Lemma~\ref{lgb_equivalence}.  Since~\(q\)
  is a quotient map, there is \(T \in\bdd_B\)
  with \(q(T) = S_2\).
  We may choose~\(T\)
  with \(\dvgen S^j \subseteq T\).
  For each \(x,y \in T\),
  we have \(q(x\cdot y) \in S_2 \cdot S_2 \subseteq S_2 = q(T)\).
  Hence there is \(\omega(x,y) \in T\)
  with \(x\cdot y - \omega(x,y) \in i(A)\).  Let
  \[
  \Omega \defeq \setgiven{x\cdot y - \omega(x,y)}{x,y\in T}.
  \]
  This is contained in \(T^2 - T\).
  So \(\Omega\in\bdd_B\).
  And \(T^2 \subseteq T + \Omega\).
  By construction, \(\Omega\)
  is also contained in~\(i(A)\).
  Since~\(i\)
  is a bornological embedding, \(i^{-1}(\Omega)\)
  is bounded in~\(A\).
  Since~\(A\)
  is semi-dagger, we have \(\varrho_1(i^{-1}(\Omega); \bdd_A) = 1\).
  So \(\sum_{n=1}^\infty (\dvgen \cdot \Omega)^n\)
  is bounded.  Thus the subset
  \[
  U \defeq \sum_{n=1}^\infty (\dvgen \cdot \Omega)^n
  + \sum_{n=0}^\infty T \cdot (\dvgen \cdot \Omega)^n
  = \sum_{n=1}^\infty (\dvgen \cdot \Omega)^n
  + T + \sum_{n=1}^\infty T \cdot (\dvgen \cdot \Omega)^n
  \]
  of~\(B\)
  is bounded.  Using \(T^2 \subseteq T + \Omega\),
  we prove that \(\dvgen T\cdot U \subseteq U\).
  Hence \((\dvgen T)^n\cdot U \subseteq U\)
  for all \(n\in\N_{\ge1}\)
  by induction.  Since \(T \subseteq U\),
  this implies \(\sum_{l=1}^\infty \dvgen^l T^{l+1} \subseteq U\).
  Hence
  \(\sum_{l=2}^\infty (\dvgen T)^l = \dvgen \cdot \sum_{l=1}^\infty
  \dvgen^l T^{l+1} \subseteq \dvgen U\).
  Therefore, \(\sum_{l=1}^\infty (\dvgen T)^l\)
  is bounded.  Since \(\dvgen^2 S^j \subseteq \dvgen T\),
  it follows that \( \sum_{l=1}^\infty (\dvgen^2 S^j)^l\)
  is bounded for all~\(j\), as desired.
\end{proof}

\begin{proposition}[\cite{Cortinas-Cuntz-Meyer-Tamme:Nonarchimedean}*{Lemma 3.1.12}]
  \label{pro:completion_semidagger}
  If~\(A\) is semi-dagger, then so is its completion~\(\comb{A}\).
\end{proposition}

Let~\(\comling{A}\)
be the completion of~\(\ling{A}\).
This algebra is both complete and semi-dagger by
Proposition~\ref{pro:completion_semidagger}.  The canonical bounded
homomorphism \(A\to \comling{A}\)
is the universal arrow from~\(A\)
to a complete semi-dagger algebra, that is, any bounded homomorphism
\(A\to B\)
for a complete semi-dagger algebra~\(B\)
factors uniquely through it.  This follows immediately from the
universal properties of the linear growth bornology and the
completion.

\section{Bornological torsion-freeness}
\label{sec:born_tf}

Let~\(M\)
be a bornological \(\dvr\)\nb-module.
The bounded linear map \(\dvgen_M\colon M\to M\),
\(m\mapsto \dvgen\cdot m\), is defined in~\eqref{eq:mult_dvgen}.

\begin{definition}
  A bornological \(\dvr\)\nb-module~\(M\)
  is \emph{bornologically torsion-free} if~\(\dvgen_M\)
  is a bornological embedding.  Equivalently, \(\dvgen \cdot m = 0\)
  for \(m\in M\)
  only happens for \(m = 0\)
  and any bounded subset of~\(M\)
  that is containd in \(\dvgen \cdot M\)
  is of the form \(S = \dvgen \cdot T\) for some \(T \in \bdd_M\).
\end{definition}

Bornological \(\dvf\)\nb-vector
spaces are bornologically torsion-free because bornological
isomorphisms are bornological embeddings.  We are going to show
that~\(M\)
is bornologically torsion-free if and only if the canonical map
\(\iota_M\colon M\to \dvf\otimes M\)
defined in Lemma~\ref{lem:universal_to_vector_space} is a bornological
embedding.  The proof uses the following easy permanence property:

\begin{lemma}
  \label{sub_tf}
  Let~\(M\)
  be a bornological \(\dvr\)\nb-module
  and let \(N\subseteq M\)
  be a \(\dvr\)\nb-submodule
  with the subspace bornology.  If~\(M\)
  is bornologically torsion-free, then so is~\(N\).
\end{lemma}

\begin{proof}
  Let \(j\colon N\to M\)
  be the inclusion map, which is a bornological embedding by
  assumption.  Since~\(\dvgen_M\)
  is a bornological embedding, so is
  \(\dvgen_M\circ j = j\circ \dvgen_N\).
  Since~\(j\)
  is a bornological embedding, this implies that~\(\dvgen_N\)
  is a bornological embedding.  That is, \(N\)
  is bornologically torsion-free.
\end{proof}

\begin{proposition}
  \label{pro:bornological_torsion_characterisation}
  A bornological \(\dvr\)\nb-module~\(M\)
  is bornologically torsion-free if and only if the canonical map
  \(\iota_M\colon M \to \dvf\otimes M\) is a bornological embedding.
\end{proposition}

\begin{proof}
  As a bornological \(\dvf\)\nb-vector
  space, \(\dvf\otimes M\)
  is bornologically torsion-free.  Hence~\(M\)
  is bornologically torsion-free by Lemma~\ref{sub_tf} if~\(\iota_M\)
  is a bornological embedding.  Conversely, assume that~\(M\)
  is bornologically torsion-free.  The map~\(\iota_M\)
  is injective because~\(M\)
  is algebraically torsion-free.  It remains to show that a
  subset~\(S\)
  of~\(M\)
  is bounded if \(\iota_M(S)\subseteq \dvf\otimes M\)
  is bounded.  If~\(\iota_M(S)\)
  is bounded, then it is contained in \(\dvgen^{-k}\cdot \dvr\otimes T\)
  for some \(k\in\N\)
  and some \(T\in\bdd_M\).
  Equivalently, \(\dvgen_M^k(S) = \dvgen^k\cdot S\)
  is bounded in~\(M\).
  Since~\(\dvgen_M\)
  is a bornological embedding, induction shows that
  \(\dvgen_M^k\colon M\to M\),
  \(m\mapsto \dvgen^k\cdot m\),
  is a bornological embedding as well.  So the boundedness of
  \(\dvgen_M^k(S)\) implies that~\(S\) is bounded.
\end{proof}

\begin{proposition}
  \label{universal_torsionfree}
  Let \(\torf{M} \defeq \iota_M(M)\subseteq \dvf\otimes M\)
  equipped with the subspace bornology and the surjective bounded
  linear map \(\iota_M\colon M\to \torf{M}\).
  This is the universal arrow from~\(M\)
  to a bornologically torsion-free module, that is, any bounded linear
  map \(f\colon M\to N\)
  into a bornologically torsion-free module~\(N\)
  factors uniquely through a bounded linear map
  \(f^\#\colon \torf{M} \to N\).
\end{proposition}

\begin{proof}
  Since \(\dvf\otimes M\)
  is bornologically torsion-free as a bornological \(\dvf\)\nb-vector
  space, \(\torf{M}\)
  is bornologically torsion-free as well by Lemma~\ref{sub_tf}.  We
  prove the universality of the canonical map
  \(\iota_M\colon M\to \torf{M}\).
  Let~\(N\)
  be a bornologically torsion-free \(\dvr\)\nb-module
  and let \(f\colon M \to N\)
  be a bornological \(\dvr\)\nb-module
  map.  Then \(N \injto \dvf\otimes N\)
  is a bornological embedding by
  Proposition~\ref{pro:bornological_torsion_characterisation}, and we
  may compose to get a bounded \(\dvr\)\nb-linear
  map \(M \to \dvf\otimes N\).
  By Lemma~\ref{lem:universal_to_vector_space}, there is a unique
  bounded \(\dvf\)\nb-linear
  map \(f'\colon \dvf\otimes M \to \dvf\otimes N\)
  with \(f'(\iota_M(m)) = f(m)\)
  for all \(m\in M\).
  Since \(f'(\iota_M(M)) \subseteq N\),
  \(f'\)
  maps the submodule \(\torf{M}\subseteq \dvf\otimes M\)
  to the submodule \(N\subseteq \dvf\otimes N\).
  The restricted map \(f^\#\colon \torf{M} \to N\)
  is bounded because both submodules carry the subspace bornology.
  This is the required factorisation of~\(f\).
  It is unique because \(\iota_M\colon M\to\torf{M}\) is surjective.
\end{proof}

We have seen that being bornologically torsion-free is hereditary for
submodules.  The obvious counterexample
\(\resf = \dvr \mathbin{/} \dvgen \dvr\)
shows that it cannot be hereditary for quotients.  Next we show that
it is hereditary for extensions:

\begin{theorem}
  \label{the:extension_tf}
  Let \(M' \xrightarrow{i} M \xrightarrow{q} M''\)
  be an extension of bornological \(\dvr\)\nb-modules.
  If \(M'\)
  and~\(M''\) are bornologically torsion-free, then so is~\(M\).
\end{theorem}

\begin{proof}
  The exactness of the sequence
  \(0 \to \ker \dvgen_{M'} \to \ker \dvgen_M \to \ker \dvgen_{M''}\)
  shows that~\(\dvgen_M\)
  is injective.  Let \(S\in \bdd_M\)
  be contained in~\(\dvgen M\).
  We want a bounded subset \(S' \in \bdd_M\)
  with \(\dvgen \cdot S' = S\).
  We have
  \(q(S) \subseteq q(\dvgen\cdot M) \subseteq \dvgen\cdot M''\),
  and \(q(S)\in \bdd_{M''}\)
  because~\(q\)
  is bounded.  Since~\(M''\)
  is bornologically torsion-free, there is \(T''\in\bdd_{M''}\)
  with \(\dvgen\cdot T'' = q(S)\).
  Since~\(q\)
  is a bornological quotient map, there is \(T\in \bdd_M\)
  with \(q(T) = T''\).
  Thus \(q(\dvgen\cdot T) = q(S)\).
  So for any \(x\in S\)
  there is \(y\in T\)
  with \(q(\dvgen\cdot y) = q(x)\).
  Since \(i = \ker(q)\),
  there is a unique \(z\in M'\)
  with \(x-\dvgen y = i(z)\).
  Let~\(T'\)
  be the set of these~\(z\).
  Since \(x\in \dvgen\cdot M\)
  by assumption and~\(M''\)
  is torsion-free, we have \(z \in \dvgen \cdot M'\).
  So \(T'\subseteq \dvgen\cdot M'\).
  And~\(T'\)
  is bounded because \(T' \subseteq i^{-1}(S-\dvgen\cdot T)\)
  and~\(i\)
  is a bornological embedding, Since~\(M'\)
  is bornologically torsion-free, there is a bounded subset
  \(U'\in \bdd_{M'}\)
  with \(\dvgen \cdot U' = T'\).
  Then
  \(S\subseteq \dvgen\cdot T + i(\dvgen\cdot U') = \dvgen\cdot (T +
  i(U'))\).
\end{proof}

Next we prove that bornological torsion-freeness is inherited by
completions:

\begin{theorem}
  \label{the:completion_tf}
  If~\(M\)
  is bornologically torsion-free, then so is its bornological
  completion~\(\comb{M}\).
\end{theorem}

The proof requires some preparation.  We must look closely at the
construction of completions of bornological \(\dvr\)\nb-modules.

\begin{proposition}[\cite{Cortinas-Cuntz-Meyer-Tamme:Nonarchimedean}*{Proposition~2.15}]
  \label{completions_exist}
  Let~\(M\)
  be a bornological \(\dvr\)\nb-module.
  A completion of~\(M\)
  exists and is constructed as follows.  Write \(M = \varinjlim M_i\)
  as an inductive limit of the directed set of its bounded
  \(\dvr\)\nb-submodules.
  Let~\(\coma{M_i}\)
  denote the \(\dvgen\)\nb-adic
  completion of~\(M_i\).
  These form an inductive system as well, and
  \(\comb{M} \cong (\varinjlim \coma{M_i} ) \bigm/ \overline{\{0\}}\)
  is the separated quotient of their bornological inductive limit.

  The completion functor commutes with colimits, that is, the
  completion of a colimit of a diagram of bornological
  \(\dvr\)\nb-modules
  is the separated quotient of the colimit of the diagram of
  completions.
\end{proposition}

Since taking quotients may create torsion, the
information above is not yet precise enough to show that completions
inherit bornological torsion-freeness.  This requires some more work.
First we write~\(M\)
in a certain way as an inductive limit, using that it is
bornologically torsion-free.  For a bounded submodule~\(S\)
in~\(M\), let
\begin{align*}
  \dvgen^{-n} S
  &\defeq \setgiven{x\in M}{\dvgen^n \cdot x \in S}
  \subseteq M\\
  M_S &\defeq \bigcup_{n\in\N} \dvgen^{-n} S \subseteq M.
\end{align*}
The \emph{gauge semi-norm} of~\(S\)
is defined by
\(\norm{x}_S \defeq \inf \setgiven{\varepsilon^n}{x\in\dvgen^n S}\),
where \(\varepsilon=\abs{\dvgen}\)
(see \cite{Cortinas-Cuntz-Meyer-Tamme:Nonarchimedean}*{Example~2.4}).
A subset is bounded for this semi-norm if and only if it is contained
in~\(\dvgen^{-n} S\)
for some \(n\in\N\).
Since~\(M\)
is bornologically torsion-free, \(\dvgen^{-n} S\in\bdd_M\)
for \(n\in\N\).
So subsets that are bounded in the gauge semi-norm on~\(M_S\)
are bounded in~\(M\).
If \(S\subseteq T\),
then \(M_S\subseteq M_T\)
and the inclusion is contracting and hence bounded.  The
bornological inductive limit of this inductive system is naturally
isomorphic to~\(M\)
because any bounded subset of~\(M\)
is bounded in~\(M_S\)
for some bounded submodule \(S\subseteq M\)
(compare the proof of
\cite{Cortinas-Cuntz-Meyer-Tamme:Nonarchimedean}*{Proposition~2.5}).

The bornological completion~\(\comb{M_S}\)
of~\(M_S\)
as a bornological \(\dvr\)\nb-module
is canonically isomorphic to its Hausdorff completion as a semi-normed
\(\dvr\)\nb-module.
We call this a \emph{Banach \(\dvr\)\nb-module}.
Both completions are isomorphic to the increasing union of the
\(\dvgen\)\nb-adic
completions \(\coma{\dvgen^{-n}S}\).
If \(S\subseteq T\),
then \(M_S\subseteq M_T\)
and this inclusion is norm-contracting.  So we get an induced
contractive linear map \(i_{T,S}\colon \comb{M_S} \to \comb{M_T}\).
This map need not be injective any more (see
\cite{Cortinas-Cuntz-Meyer-Tamme:Nonarchimedean}*{Example~2.15}).
Hence the canonical maps \(i_{\infty,S}\colon \comb{M_S} \to \comb{M}\)
need not be injective.
The bornological completion commutes with (separated) inductive limits
by Proposition~\ref{completions_exist}.  So the completion of~\(M\)
is isomorphic to the separated quotient of the colimit of the
inductive system formed by the Banach \(\dvr\)\nb-modules
\(\comb{M_S}\)
and the norm-contracting maps \(i_{T,S}\)
for \(S\subseteq T\).

\begin{lemma}
  \label{lem:kernel_to_completion}
  The submodules
  \[
  Z_S\defeq \ker(i_{\infty,S}\colon \comb{M_S} \to \comb{M})
  =  i_{\infty,S}^{-1}(\{0\}) \subseteq \comb{M_S}
  \]
  are norm-closed and satisfy \(i_{T,S}^{-1}(Z_T) = Z_S\)
  if \(S\subseteq T\).
  They are minimal with these properties in the sense that if
  \(L_S \subseteq \comb{M_S}\)
  are norm-closed and satisfy \(i_{T,S}^{-1}(L_T) = L_S\)
  for \(S\subseteq T\),
  then \(Z_S \subseteq L_S\)
  for all bounded \(\dvr\)\nb-submodules \(S\subseteq M\).
\end{lemma}

\begin{proof}
  The property \(i_{T,S}^{-1}(Z_T) = Z_S\)
  is trivial.  The map~\(i_{\infty,S}\)
  is bounded and hence preserves convergence of sequences.
  Since~\(\comb{M}\)
  is separated, the subset \(\{0\}\subseteq \comb{M}\)
  is bornologically closed.  Therefore, its preimage~\(Z_S\)
  in~\(\comb{M_S}\)
  is also closed.
  Let~\((L_S)\)
  be any family of closed submodules with \(i_{T,S}^{-1}(L_T) = L_S\).
  The quotient seminorm on \(\comb{M_S}/L_S\)
  is again a norm because~\(L_S\)
  is closed.  And \(\comb{M_S}/L_S\)
  inherits completeness from~\(\comb{M_S}\)
  by Theorem~\ref{the:extension_complete}.
  If \(S\subseteq T\),
  then~\(\iota_{T,S}\)
  induces an injective map
  \(i_{T,S}'\colon \comb{M_S}/L_S \to \comb{M_T}/L_T\)
  because \(L_S = i_{T,S}^{-1}(L_T)\).
  Hence the colimit of the inductive system
  \((\comb{M_S}/L_S,i_{T,S}')\)
  is like a directed union of subspaces, and each \(\comb{M_S}/L_S\)
  maps faithfully into it.  Thus this colimit is separated.  It is
  even complete because each \(\comb{M_S}/L_S\)
  is complete.  Hence the map from~\(M\)
  to this colimit induces a map on the completion~\(\comb{M}\).
  This implies \(Z_S \subseteq L_S\).
\end{proof}

Next we link~\(\comb{M}\)
to the \(\dvgen\)\nb-adic
completion \(\coma{M} \defeq \varprojlim M/\dvgen^j M\).
Equip the quotients \(M/\dvgen^j M\)
with the quotient bornology.  Since
\(\dvgen^j\cdot (M/\dvgen^j) = 0\),
any Cauchy sequence in \(M/\dvgen^j M\)
is eventually constant.  So each \(M/\dvgen^j M\)
is complete.  Hence the quotient map \(M\to M/\dvgen^j M\)
induces a bounded \(\dvr\)\nb-module
homomorphism \(\comb{M} \to M/\dvgen^j M\).
Putting them all together gives a map \(\comb{M} \to \coma{M}\),
which is bounded if we give~\(\coma{M}\)
the projective limit bornology.

Let \(S\subseteq M\)
be a bounded \(\dvr\)\nb-submodule
and let \(j\in\N\).
We have defined the submodules~\(M_S\)
so that \(M_S \cap \dvgen^j M = \dvgen^j M_S\).
That is, the map \(M_S/ \dvgen^j M_S \to M/\dvgen^j M\)
is injective.  Since~\(M_S\)
is dense in its norm-completion~\(\comb{M_S}\),
we have \(\comb{M_S} = M_S + \dvgen^j \coma{S}\)
and hence \(\comb{M_S} = M_S + \dvgen^j \comb{M_S}\).
Thus the inclusion \(M_S \to \comb{M_S}\) induces an isomorphism
\[
M_S/\dvgen^j M_S \cong \comb{M_S}/\dvgen^j \comb{M_S}.
\]
Letting~\(j\)
vary, we get an injective map \(\coma{M_S} \to \coma{M}\)
and an isomorphism between the \(\dvgen\)\nb-adic
completions of \(M_S\) and~\(\comb{M_S}\).

\begin{proof}[Proof of Theorem~\textup{\ref{the:completion_tf}}]
  For each bounded \(\dvr\)\nb-submodule
  \(S\subseteq M\),
  define
  \(L_S \defeq \bigcap_{j\in\N} \dvgen^j\cdot \comb{M_S} \subseteq
  \comb{M_S}\).
  This is the kernel of the canonical map to the \(\dvgen\)\nb-adic
  completion of~\(\comb{M_S}\).
  The completion~\(\comb{M_S}\)
  is torsion-free because it carries a norm.  Hence~\(L_S\)
  is also the largest \(\dvf\)\nb-vector
  space contained in~\(\comb{M_S}\).
  The subspace~\(L_S\)
  is closed because the maps
  \(\comb{M_S} \to \comb{M_S}/\dvgen^j \comb{M_S}\)
  for \(j\in\N\)
  are bounded and their target spaces are separable, even complete.

  Let \(S\subseteq T\).
  The maps \(M_S/\dvgen^j M_S \to M_T/\dvgen^j M_T\)
  are injective for all \(j\in\N\),
  and \(\comb{M_S}/\dvgen^j \comb{M_S} \cong M_S/\dvgen^j M_S\),
  \(\comb{M_T}/\dvgen^j \comb{M_T} \cong M_T/\dvgen^j M_T\).
  So~\(i_{T,S}\)
  induces an injective map
  \(\comb{M_S}/\dvgen^j \comb{M_S} \to \comb{M_T}/\dvgen^j
  \comb{M_T}\).
  This implies
  \(i_{T,S}^{-1}(\dvgen^j \comb{M_T}) = \dvgen^j \comb{M_S}\)
  for all \(j\in\N\) and then \(i_{T,S}^{-1}(L_T) = L_S\).

  By Lemma~\ref{lem:kernel_to_completion}, the kernel
  \(Z_S = \ker(i_{\infty,S})\)
  is contained in~\(L_S\)
  for all~\(S\).
  Since \(\dvgen_{L_S}\)
  is a bornological isomorphism, the subsets
  \(\dvgen\cdot Z_S \subseteq Z_S\)
  are also bornologically closed, and they satisfy
  \(i_{T,S}^{-1}(\dvgen\cdot Z_T) = \dvgen\cdot i_{T,S}^{-1}(Z_T)
  = \dvgen\cdot Z_S\).
  Hence \(Z_S \subseteq \dvgen\cdot Z_S\)
  for all~\(S\)
  by Lemma~\ref{lem:kernel_to_completion}.  Thus \(Z_S\subseteq L_S\)
  is a \(\dvf\)\nb-vector
  subspace in~\(\comb{M_S}\).
  So the quotient \(\comb{M_S}/Z_S\)
  is still bornologically torsion-free.  And any element of
  \(\comb{M_S}/Z_S\)
  that is divisible by~\(\dvgen^j\)
  lifts to an element in \(\dvgen^j\cdot \comb{M_S}\).

  Any bounded subset of~\(\comb{M}\)
  is contained in \(i_{\infty,S}(\coma{S})\)
  for some bounded \(\dvr\)\nb-submodule
  \(S\subseteq M\),
  where we view~\(\coma{S}\)
  as a subset of~\(\comb{M_S}\).
  Let \(j\in\N\).
  To prove that~\(\comb{M}\)
  is bornologically torsion-free, we must show that
  \(\dvgen^{-j} i_{\infty,S}(\coma{S})\)
  is bounded.  Let \(x\in\comb{M}\)
  satisfy \(\dvgen^j x \in i_{\infty,S}(\coma{S})\).
  We claim that \(x=i_{\infty,S}(y)\)
  for some \(y\in \comb{M_S}\)
  with \(\dvgen^j y \in\coma{S}\).
  This implies that \(\dvgen^{-j}\cdot i_{\infty,S}(\coma{S})\)
  is bounded in~\(\comb{M}\).
  It remains to prove the claim.  There are a bounded
  \(\dvr\)\nb-submodule
  \(T\subseteq M\)
  and \(z\in \comb{M_T}\)
  with \(x=i_{\infty,T}(z)\).
  We may replace~\(T\)
  by \(T+S\)
  to arrange that \(T\supseteq S\).
  Let \(w\in \coma{S}\)
  satisfy \(\dvgen^j x = i_{\infty,S}(w)\).
  This is equivalent to
  \(\dvgen^j z - i_{T,S}(w) \in \ker i_{\infty,T} = Z_T\).
  Since~\(Z_T\)
  is a \(\dvf\)\nb-vector
  space, there is \(z_0\in Z_T\)
  with \(\dvgen^j z - i_{T,S}(w) = \dvgen^j z_0\).
  Since \(x=i_{\infty,T}(z-z_0)\),
  we may replace~\(z\)
  by \(z-z_0\)
  to arrange that \(\dvgen^j z = i_{T,S}(w)\).
  Since \(i_{T,S}^{-1}(\dvgen^j \comb{M_T}) = \dvgen^j \comb{M_S}\),
  there is \(y\in \comb{M_S}\)
  with \(\dvgen^j\cdot y = w\).
  Then \(\dvgen^j z = \dvgen^j i_{T,S}(y)\).
  This implies \(z = i_{T,S}(y)\)
  because~\(\comb{M_T}\) is torsion-free.  This proves the claim.
\end{proof}

\begin{proposition}
  \label{pro:born_tf_linear_completion}
  Let~\(M\)
  be a bornologically torsion-free bornological \(\dvr\)\nb-module.
  Then \(\dvf\otimes \comb{M} \cong \comb{\dvf\otimes M}\)
  with an isomorphism compatible with the canonical maps from~\(M\)
  to both spaces.
\end{proposition}

\begin{proof}
  The canonical map \(M\to \comb{M}\)
  is the universal arrow from~\(M\)
  to a complete \(\dvr\)\nb-module.
  The canonical map \(M\to \dvf\otimes M\)
  is the universal arrow from~\(M\)
  to a bornological \(\dvf\)\nb-vector
  space by Lemma~\ref{lem:universal_to_vector_space}.  Since
  \(\dvf\otimes \comb{M}\)
  is again complete, the canonical map \(M\to \dvf\otimes \comb{M}\)
  is the universal arrow from~\(M\)
  to a complete bornological \(\dvf\)\nb-vector
  space.  The completion \(\comb{\dvf\otimes M}\)
  is also a bornological \(\dvf\)\nb-vector
  space.  The canonical map \(M\to \comb{\dvf\otimes M}\)
  is another universal arrow from~\(M\)
  to a complete bornological \(\dvf\)\nb-vector
  space.  Since the universal property determines its target uniquely
  up to canonical isomorphism, there is a unique isomorphism
  \(\dvf\otimes \comb{M} \cong \comb{\dvf\otimes M}\)
  that makes the following diagram commute:
  \[
  \begin{tikzcd}[column sep=small, row sep=small]
    \dvf\otimes \comb{M} \arrow[rr, "\cong"] && \comb{\dvf\otimes M}\\
    & M \arrow[ul] \arrow[ur]
  \end{tikzcd}\qedhere
  \]
\end{proof}

\begin{corollary}
  \label{cor:tf_embedding}
  If~\(M\)
  is bornologically torsion-free, then the canonical map
  \(\comb{M} \to \comb{\dvf\otimes M}\) is a bornological embedding.
\end{corollary}

\begin{proof}
  Use the isomorphism
  \(\dvf\otimes \comb{M} \cong \comb{\dvf\otimes M}\)
  to replace the canonical map \(\comb{M} \to \comb{\dvf\otimes M}\)
  by the canonical map \(\comb{M} \to \dvf\otimes \comb{M}\).
  This is a bornological embedding if and only if~\(\comb{M}\)
  is bornologically torsion-free by
  Proposition~\ref{pro:bornological_torsion_characterisation}.  And
  this is true by Theorem~\ref{the:completion_tf}.
\end{proof}

Finally, we show that being bornologically torsion-free is compatible
with linear growth bornologies:

\begin{proposition}
  \label{pro:lgb_inherits_tf}
  If~\(A\)
  is a bornologically torsion-free \(\dvr\)\nb-algebra,
  then so is~\(\ling{A}\).
\end{proposition}

\begin{proof}
  Let \(S \subseteq \dvgen \cdot A\)
  be bounded in~\(\ling{A}\).
  Then there is \(T\in\bdd_A\)
  with \(S \subseteq T_1 \defeq \sum_{i=0}^\infty \dvgen^i T^{i+1}\)
  by Lemma~\ref{equivalent_semidagger}.
  The subset \(T_2 \defeq \sum_{i=0}^\infty \dvgen^i T^{i+2}\)
  also has linear growth.  And
  \[
  T_1 = T + \sum_{i=1}^\infty \dvgen^i T^{i+1}
  = T + \sum_{i=0}^\infty \dvgen^{i+1} T^{i+2}
  = T + \dvgen T_2.
  \]
  Since~\(T\)
  is bounded in~\(A\)
  and~\(A\)
  is bornologically torsion-free,
  \(\dvgen^{-1} \cdot T \defeq \setgiven{x\in A}{\dvgen \cdot x \in
    T}\)
  is also bounded.  We have
  \(\dvgen^{-1}S \subseteq \dvgen^{-1} T_1 \subseteq \dvgen^{-1}\cdot
  T + T_2\).  This is bounded in~\(\ling{A}\).
\end{proof}

The following proposition answers a question by Guillermo
Corti\~nas:

\begin{proposition}
  \label{pro:tensor_tf}
  Let \(M\)
  and~\(N\)
  be bornological \(\dvr\)\nb-modules.
  If \(M\)
  and~\(N\)
  are bornologically torsion-free, then so is \(M\otimes N\)
  with the tensor product bornology.
\end{proposition}

\begin{proof}
  Since \(M\)
  and~\(N\)
  are torsion-free, so is \(M\otimes N\),
  that is, multiplication by~\(\dvgen\)
  on \(M\otimes N\)
  is injective.  Let \(U\subseteq M\otimes N\)
  be a subset such that \(\dvgen U\)
  is bounded.  We must show that~\(U\)
  is bounded.  By the definition of the tensor product bornology,
  there are bounded \(\dvr\)\nb-submodules
  \(S\subseteq M\),
  \(T\subseteq N\)
  such that \(\dvgen\cdot U \subseteq S\otimes T\).
  Define
  \[
  \dvgen^{-1} S \defeq \setgiven{x\in M}{\dvgen x\in S},\qquad
  \dvgen^{-1} T \defeq \setgiven{y\in N}{\dvgen y\in T}.
  \]
  These subsets are bounded because
  \(\dvgen\cdot (\dvgen^{-1} S) \subseteq S\)
  and \(\dvgen\cdot (\dvgen^{-1} T) \subseteq T\)
  and \(M\)
  and~\(N\)
  are bornologically torsion-free. 
  We claim that \(U\subseteq \dvgen^{-1} S \otimes \dvgen^{-1} T\).
  This shows that~\(U\)
  is bounded.

  Let \(u\in U\).
  We may write \(u = \sum_{j=1}^N x_j \otimes y_j\)
  with \(x_j\in M\), \(y_j\in N\).
  Since \(\dvgen \cdot u\in S\otimes T\),
  we may write \(\dvgen u = \sum_{k=1}^M z_k\otimes w_k\)
  with \(z_k\in S\),
  \(w_k\in T\).
  Let \(A\subseteq M\)
  and \(B\subseteq N\)
  be the \(\dvr\)\nb-submodules
  generated by the elements \(x_j,z_k\)
  and \(y_j,w_k\),
  respectively.  These submodules are finitely generated and
  torsion-free, hence free.
  And the canonical map
  \(A\otimes B \to M \otimes N\)
  is injective.
  The submodules \(A\cap S\)
  and \(B\cap T\)
  are also free.  Any \(\dvr\)\nb-module homomorphism between
  finitely generated free \(\dvr\)\nb-modules
  may be brought into diagonal form with entries in
  \(\{\dvgen^\N\} \cup\{0\}\)
  along the diagonal by choosing appropriate bases in the
  \(\dvr\)\nb-modules.
  Therefore, there are \(\dvr\)\nb-module
  bases \(a_1,\dotsc,a_n\)
  and \(b_1,\dotsc,b_m\)
  of \(A\)
  and~\(B\),
  respectively,
  and \(1\le n' \le n\),
  \(1\le m' \le m\),
  \(0 \ge \alpha_1 \ge \alpha_2 \ge \dotsb \ge \alpha_{n'}\),
  \(0 \ge \beta_1 \ge \beta_2 \ge \dotsb \ge \beta_{m'}\),
  such that \(\dvgen^{\alpha_i}\cdot a_i\)
  and \(\dvgen^{\beta_j}\cdot b_j\)
  for \(1 \le i \le n'\)
  and \(1 \le j \le m'\)
  are \(\dvr\)\nb-module
  bases of \(A\cap S\)
  and \(B\cap T\), respectively.
  We may write \(u\in A \otimes B\)
  uniquely in this basis as
  \(u = \sum_{i,j} u_{i,j} a_i \otimes b_j\)
  with \(u_{i,j}\in \dvr\).
  By assumption,
  \(\dvgen\cdot u = \sum_{k=1}^M y_k \otimes w_k \in (S\cap A)
  \otimes (T\cap B)\).
  By construction, the elements
  \(\dvgen^{\alpha_i+\beta_j} a_i \otimes b_j\)
  form a \(\dvr\)\nb-module
  basis of \((S\cap A) \otimes (T\cap B)\).
  Since the coefficients of~\(\dvgen u\)
  in the basis \(a_i \otimes b_j\)
  of \(A\otimes B\)
  are unique, it follows that \(u_{i,j}=0\)
  if \(i>n'\) or \(j>m'\),
  and \(\dvgen u_{i,j} \in \dvgen^{\alpha_i+\beta_j} \dvr\)
  for \(1\le i \le n'\)
  and \(1\le j \le m'\).
  Hence~\(u\)
  is a \(\dvr\)\nb-linear
  combination of
  \(\dvgen^{(\alpha_i-1)_+} a_i \otimes \dvgen^{(\beta_j-1)_+} b_j\)
  for \(1 \le i \le n'\), \(1 \le j \le m'\),
  where \(n_+ \defeq \max \{n,0\}\).
  Since \(\dvgen^{(\alpha_i-1)_+} a_i \in \dvgen^{-1} S\),
  \(\dvgen^{(\alpha_j-1)_+} b_j \in \dvgen^{-1} T\),
  this implies \(u \in \dvgen^{-1} S \otimes \dvgen^{-1} T\).
  Since \(u\in U\)
  was arbitrary, we get
  \(U \subseteq \dvgen^{-1} S \otimes \dvgen^{-1} T\).
\end{proof}

Theorem~\ref{the:completion_tf} and Proposition~\ref{pro:tensor_tf}
imply that bornological torsion-freeness for complete bornological
\(\dvr\)\nb-modules
is hereditary for completed tensor products.

\section{Dagger algebras}
\label{sec:dagger}

\begin{definition}
  A \emph{dagger algebra} is a complete, bornologically torsion-free,
  semi-dagger algebra.
\end{definition}

\begin{theorem}
  Let \(A \xrightarrow{i} B \xrightarrow{p} C\)
  be an extension of bornological \(\dvr\)\nb-algebras.
  If~\(A\) and~\(C\) are dagger algebras, so is~\(B\).
\end{theorem}

\begin{proof}
  All three properties defining dagger algebras are hereditary for
  extensions by Theorems \ref{the:extension_complete} (because~\(C\)
  is torsion-free), \ref{the:extension_lgb}
  and~\ref{the:extension_tf}.
\end{proof}

We have already seen that there are universal arrows
\(A\to \torf{A} \subseteq \dvf\otimes A\),
\(A\to\ling{A}\),
\(A\to \comb{A}\)
from a bornological algebra~\(A\)
to a bornologically torsion-free algebra, to a semi-dagger algebra, and to
a complete bornological algebra, respectively.  We now combine them to
a universal arrow to a dagger algebra:

\begin{theorem}
  \label{dagger_completion}
  Let~\(A\)
  be a bornological algebra.  Then the canonical map
  from~\(A\)
  to \(A^\updagger \defeq \comling{(\torf{A})}\)
  is the universal arrow from~\(A\)
  to a dagger algebra.  That is, any bounded algebra homomorphism
  from~\(A\)
  to a dagger algebra factors uniquely through~\(A^\updagger\).
  If~\(A\)
  is already bornologically torsion-free, then
  \(A^\updagger \cong \comling{A}\).
\end{theorem}

\begin{proof}
  The bornological algebra~\(A^\updagger\)
  is complete by construction.  It is semi-dagger by
  Proposition~\ref{pro:completion_semidagger}.  And it is
  bornologically torsion-free by Proposition~\ref{pro:lgb_inherits_tf}
  and Theorem~\ref{the:completion_tf}.  So it is a dagger algebra.
  Let~\(B\)
  be a dagger algebra.  A bounded homomorphism \(A\to B\)
  factors uniquely through a bounded homomorphism \(\torf{A}\to B\)
  by Proposition~\ref{universal_torsionfree} because~\(B\)
  is bornologically torsion-free.  This factors uniquely through a
  bounded homomorphism \(\ling{(\torf{A})}\to B\)
  because~\(B\)
  is semi-dagger.  And this factors uniquely through a bounded
  homomorphism \(\comling{(\torf{A})}\to B\)
  because~\(B\)
  is complete.  So~\(A^\updagger\)
  has the asserted universal property.  If~\(A\)
  is bornologically torsion-free, then \(A\cong \torf{A}\)
  and hence \(A^\updagger \cong \comling{A}\).
\end{proof}

\begin{definition}
  We call~\(A^\updagger\)
  the \emph{dagger completion} of the bornological
  \(\dvr\)\nb-algebra~\(A\).
\end{definition}

\section{Dagger completions of monoid algebras}
\label{sec:dagger_monoid}

As a simple illustration, we describe the dagger completions of monoid
algebras.  The monoid algebra of~\(\N^j\)
is the algebra of polynomials in \(j\)~variables,
and its dagger completion is the Monsky--Washnitzer algebra of
overconvergent power series equipped with a canonical bornology
(see~\cite{Cortinas-Cuntz-Meyer-Tamme:Nonarchimedean}).  The case of
general monoids is similar.

The monoid algebra~\(\dvr[S]\)
of~\(S\)
over~\(\dvr\)
is defined by its universal property: if~\(B\)
is a unital \(\dvr\)\nb-algebra,
then there is a natural bijection between algebra homomorphisms
\(\dvr[S] \to B\)
and monoid homomorphisms \(S \to (B,\cdot)\)
into the multiplicative monoid of~\(B\).
More concretely, \(\dvr[S]\)
is the free \(\dvr\)\nb-module
with basis~\(S\)
or, equivalently, the \(\dvr\)\nb-module
of formal linear combinations of the form
\[
\sum_{s\in S} x_s \delta_s,\qquad
x_s \in \dvr,\ s\in S,
\]
with \(x_s = 0\)
for all but finitely many~\(s\),
and equipped with the multiplication
\[
\sum_{s\in S} x_s \delta_s * \sum_{t\in S} y_t \delta_t
= \sum_{s,t\in S} x_s y_t \delta_{s\cdot t}.
\]
We give~\(\dvr[S]\)
the fine bornology.  Then it has an analogous universal property in
the category of bornological \(\dvr\)\nb-algebras.
So the dagger completion \(\dvr[S]^\updagger\)
is a dagger algebra with the property that bounded algebra
homomorphisms \(\dvr[S]^\updagger\to B\)
for a dagger algebra~\(B\)
are in natural bijection with monoid homomorphisms
\(S\to (B,\cdot)\).

Assume first that~\(S\)
has a finite generating set~\(F\).
Let \(F^n\subseteq S\)
be the set of all words \(s_1\dotsm s_k\)
with \(s_1,\dotsc,s_k\in F\)
and \(k\le n\).
This gives an increasing filtration on~\(S\)
with \(F^0 = \{1\}\)
and \(S = \bigcup_{n=0}^\infty F^n\).
For \(s\in S\),
we define \(\ell(s) \in \N\)
as the smallest~\(n\)
with \(s\in F^n\).
This is the \emph{word length} generated by~\(F\).
Let \(\dvr[F^n] \subseteq \dvr[S]\)
be the free \(\dvr\)\nb-submodule
of~\(\dvr[S]\)
spanned by~\(F^n\).
Any finitely generated \(\dvr\)\nb-submodule
of~\(\dvr[S]\)
is contained in~\(\dvr[F^n]\)
for some \(n\in \N\).
By Lemma~\ref{equivalent_semidagger}, a subset of~\(\dvr[S]\)
has linear growth if and only if it is contained in
\(M_n \defeq \sum_{j=0}^\infty \dvgen^j (\dvr[F^n])^{j+1}\)
for some \(n\in \N_{\ge 1}\).  That is,
\[
\ling{M} = \varinjlim M_n.
\]
Recall the valuation \(\nu \colon \dvr \to \N \cup \{\infty\}\)
defined by
\[
\nu(x) \defeq
\sup {}\setgiven{n\in \N}{x \in \dvgen^n \dvr}.
\]
By definition, the submodule~\(M_n\)
consists of all finite sums of terms~\(x_s \delta_s\)
with \(x_s \in \dvgen^j\cdot \dvr\)
and \(\ell(s) \le n(j+1)\)
for some \(j\in\N\)
or, equivalently, \(\ell(s)/n \le j+1 \le \nu(x_s)+1\).
That is, \(M_n\)
contains a finite sum \(\sum_{s\in S} x_s \delta_s\)
with \(x_s\in \dvr\)
and \(x_s=0\)
for all but finitely many \(s\in S\)
if and only if \(\nu(x_s) + 1\ge \ell(s)/n\)
for all \(s\in S\).
The \(\dvgen\)\nb-adic
completion~\(\coma{M_n}\)
of~\(M_n\)
is the set of all formal power series \(\sum_{s\in S} x_s \delta_s\)
such that \(x_s \in \dvr\),
\(\nu(x_s) + 1\ge \ell(s)/n\)
for all \(s\in S\)
and \(\lim_{\ell(s)\to\infty} \nu(x_s) +1 - \ell(s)/n = \infty\).
This implies \(x_s \to 0\)
in the \(\dvgen\)\nb-adic
norm, so that \(\coma{M_n} \subseteq \coma{\dvr[S]}\).
So the extension \(\coma{M_n} \to \coma{M_{n+1}}\)
of the inclusion map \(M_n \to M_{n+1}\)
remains injective.  Therefore, \(\varinjlim \coma{M_n}\)
is separated, and it is contained in~\(\coma{\dvr[S]}\).
Proposition~\ref{completions_exist} implies
\[
\dvr[S]^\updagger = \varinjlim \coma{M_n}.
\]

Elements of~\(\coma{\dvr[S]}\)
are formal series \(\sum_{s\in S}x_s \delta_s\)
with \(x_s\in \dvr\)
for all \(s\in S\)
and \(\lim {} \abs{x_s} = 0\).
We have seen above that such a formal series belongs to~\(\coma{M_n}\)
if and only if \(\nu(x_s) + 1\ge \ell(s)/n\)
for all \(s\in S\)
and \(\lim_{\ell(s)\to\infty} \nu(x_s) +1 - \ell(s)/n = \infty\).
If \(0<1/n<c\),
then \(\nu(x_s) + 1\ge c\ell(s)\)
implies \(\nu(x_s) + 1\ge \ell(s)/n\)
and \(\lim_{\ell(s)\to\infty} \nu(x_s) +1 - \ell(s)/n = \infty\).
Thus all \(\sum_{s\in S}x_s \delta_s \in \coma{\dvr[S]}\)
with \(\nu(x_s) + 1\ge c\ell(s)\)
belong to \(\coma{M_n}\).
Conversely, all elements of \(\coma{M_n}\)
satisfy this for \(c=1/n\).
Letting \(c\)
and~\(n\)
vary, we see that~\(\dvr[S]^\updagger\)
is the set of all \(\sum_{s\in S}x_s \delta_s\)
in~\(\coma{\dvr[S]}\)
for which there is \(c>0\)
with
\begin{equation}
  \label{eq:dagger_monoid_growth}
  \nu(x_s) + 1\ge c\ell(s)
  \qquad\text{for all }s\in S,
\end{equation}
and that a subset of~\(\dvr[S]^\updagger\)
is bounded if and only if all its elements
satisfy~\eqref{eq:dagger_monoid_growth} for the same \(c>0\).
The growth condition~\eqref{eq:dagger_monoid_growth} does not depend
on the word length function~\(\ell\)
because the word length functions for two different generating sets
of~\(S\)
are related by linear inequalities \(\ell \le a \ell'\)
and \(\ell' \le a \ell\) for some \(a>0\).

Now we drop the assumption that~\(S\)
be finitely generated.  Then we may write~\(S\)
as the increasing union of its finitely generated submonoids.  By the
universal property, the monoid algebra of~\(S\)
with the fine bornology is a similar inductive limit in the category
of bornological \(\dvr\)\nb-algebras,
and its dagger algebra is the inductive limit in the category of
dagger algebras.  Since
\(\dvr[S']^\updagger \subseteq \coma{\dvr[S']} \subseteq
\coma{\dvr[S]}\)
for any finitely generated \(S'\subseteq S\),
we may identify this inductive limit with a subalgebra of
\(\coma{\dvr[S]}\)
as well, namely, the union of \(\dvr[S']^\updagger\)
over all finitely generated submonoids \(S'\subseteq S\).
That is, \(\dvr[S]^\updagger\)
is the set of elements of~\(\coma{\dvr[S]}\)
that are supported in some finitely generated submonoid
\(S'\subseteq S\)
and that satisfy~\eqref{eq:dagger_monoid_growth} for some length
function on~\(S'\).

We may also twist the monoid algebra.  Let
\(\dvr^\times = \setgiven{x\in\dvr}{\abs{x}=1}\)
and let \(c\colon S\times S\to \dvr^\times\)
be a normalised \(2\)\nb-cocycle, that is,
\begin{equation}
  \label{eq:twist_cocycle}
  c(r,s\cdot t)\cdot c(s,t) = c(r\cdot s,t) \cdot c(r,s),
  \qquad
  c(s,1) = c(1,s) = 1
\end{equation}
for all \(r,s,t\in S\).
The \emph{\(c\)\nb-twisted
  monoid algebra of~\(S\)},
\(\dvr[S,c]\),
is the \(\dvr\)\nb-module \(\dvr[S]\) with the twisted multiplication
\begin{equation}
  \label{eq:twisted_mult}
  \sum_{s\in S} x_s \delta_s * \sum_{t\in S} y_t \delta_t
  = \sum_{s,t\in S} x_s y_t c(s,t)\cdot \delta_{s\cdot t}.
\end{equation}
The condition~\eqref{eq:twist_cocycle} is exactly what is needed to
make this associative and unital with unit~\(\delta_1\).
Since we assume~\(c\)
to have values in~\(\dvr^\times\),
the twist does not change the linear growth bornology.  Therefore, the
dagger completion~\(\dvr[S,c]^\updagger\)
consists of all infinite sums \(\sum_{s\in S} x_s \delta_s\)
that are supported in a finitely generated submonoid of~\(S\)
and satisfy the growth condition~\eqref{eq:dagger_monoid_growth}, and
a subset is bounded if and only if all its elements satisfy these two
conditions uniformly.  Only the multiplication changes and is now
given by~\eqref{eq:twisted_mult}.

\begin{example}
  Let \(S=(\Z^2,+)\)
  with the unit element~\(0\).
  Define \(c((s_1,s_2),(t_1,t_2)) \defeq \lambda^{s_2\cdot t_1}\)
  for some \(\lambda\in\dvr^\times\).
  This satisfies~\eqref{eq:twist_cocycle}.  The resulting twisted
  convolution algebra is an analogue of a noncommutative torus
  over~\(\dvr\).
  Indeed, let \(U_1 \defeq \delta_{(1,0)}\)
  and \(U_2 \defeq \delta_{(0,1)}\)
  as elements of \(\dvr[\Z^2,c]\).
  Then \(\delta_{(-1,0)} = U_1^{-1}\)
  and \(\delta_{(0,-1)} = U_2^{-1}\)
  are inverse to them, and
  \(\delta_{(s_1,s_2)} = U_1^{s_1}\cdot U_2^{s_2}\).
  So \(U_1,U_2\)
  generate~\(\dvr[\Z^2,c]\)
  as a \(\dvr\)\nb-algebra.
  They satisfy the commutation relation
  \begin{equation}
    \label{eq:nc_torus}
    U_2\cdot U_1 = \lambda \cdot U_1\cdot U_2.
  \end{equation}
  And this already dictates the multiplication table
  in~\(\dvr[\Z^2,c]\).
  The dagger completion \(\dvr[\Z^2,c]^\updagger\)
  is isomorphic as a bornological \(\dvr\)\nb-module
  to the Monsky--Washnitzer completion of the Laurent polynomial
  algebra~\(\dvr[U_1^{\pm1},U_2^{\pm1}]\),
  equipped with a twisted multiplication
  satisfying~\eqref{eq:nc_torus}.
\end{example}

\section{Dagger completions of crossed products}
\label{sec:crossed}

Let~\(A\)
be a unital, bornological \(\dvr\)\nb-algebra,
let~\(S\)
be a finitely generated monoid and let
\(\alpha \colon S \to \Endo(A)\)
be an action of~\(S\)
on~\(A\)
by bounded algebra homomorphisms.  The \emph{crossed product}
\(A\rtimes_\alpha S\)
is defined as follows.  Its underlying bornological \(\dvr\)\nb-module
is \(A \rtimes_\alpha S = \bigoplus_{s\in S} A\)
with the direct sum bornology.  So elements of \(A \rtimes_\alpha S\)
are formal linear combinations \(\sum_{s \in S} a_s \delta_s\)
with \(a_s\in A\)
and \(a_s=0\)
for all but finitely many \(s\in S\).
The multiplication on \(A \rtimes_\alpha S\) is defined by
\[
\Bigl(\sum_{s \in S}a_s \delta_s\Bigr)
\cdot \Bigl(\sum_{t \in S}b_t \delta_t\Bigr)
\defeq \sum_{s,t \in S} a_s \alpha_s(b_t) \delta_{s t}.
\]
This makes \(A\rtimes_\alpha S\)
a bornological \(\dvr\)\nb-algebra.  What is its dagger completion?

It follows easily from the universal property that defines
\(A\subseteq A\rtimes_\alpha S\) that
\[
(A\rtimes_\alpha S)^\updagger \cong (A^\updagger\rtimes_{\alpha^\updagger} S)^\updagger;
\]
here~\(\alpha^\updagger\)
is the canonical extension of~\(\alpha\)
to the dagger completion~\(A^\updagger\),
which exists because the latter is functorial for bounded algebra
homomorphisms.  Therefore, it is no loss of generality to assume
that~\(A\)
is already a dagger algebra.  It is easy to show that
\((A\rtimes S)^\updagger\)
is the inductive limit of the dagger completions
\((A\rtimes S')^\updagger\),
where~\(S'\)
runs through the directed set of finitely generated submonoids
of~\(S\).
Hence we may also assume that~\(S\)
is finitely generated to simplify.  First we consider the following
special case:

\begin{definition}
  \label{def:unif_bounded}
  The action \(\alpha\colon S\to\Endo(A)\)
  is called \emph{uniformly bounded} if any \(U\in\bdd_A\)
  is contained in an \(\alpha\)\nb-invariant
  \(T\in\bdd_A\);
  \(\alpha\)\nb-invariance
  means \(\alpha_s(T) = T\) for all \(s\in S\).
\end{definition}

If~\(T\)
is \(\alpha\)\nb-invariant,
so is the \(\dvr\)\nb-module
generated by~\(T\).
Therefore, \(\alpha\)
is uniformly bounded if and only if any bounded subset of~\(A\)
is contained in a bounded, \(\alpha\)\nb-invariant
\(\dvr\)\nb-submodule.
If~\(A\)
is complete, then the image of~\(\coma{T}\)
in~\(A\)
is also \(\alpha\)\nb-invariant
because the maps~\(\alpha_s\)
are bornological isomorphisms.  Hence we may assume in this case
that~\(T\)
in Definition~\ref{def:unif_bounded} is a bounded,
\(\alpha\)\nb-invariant
\(\dvgen\)\nb-adically complete \(\dvr\)\nb-submodule.

\begin{proposition}
  \label{pro:uniformly_bounded_induced_actions}
  Let~\(A\)
  carry a uniformly bounded action~\(\alpha\)
  of~\(S\).
  Then the induced actions on \(\comb{A}\),
  \(\torf{A}\),
  and~\(\ling{A}\)
  are uniformly bounded as well.  Hence so is the induced action
  on~\(A^\updagger\).
\end{proposition}

\begin{proof}
  If~\(\alpha\)
  is uniformly bounded, then~\(A\)
  is the bornological inductive limit of its \(\alpha\)\nb-invariant
  bounded \(\dvr\)\nb-submodules.
  The action of~\(\alpha\)
  restricts to any such submodule~\(T\)
  and then extends canonically to its \(\dvgen\)\nb-adic
  completion~\(\coma{T}\).
  Then the image of~\(\coma{T}\)
  in~\(\comb{A}\)
  is \(S\)\nb-invariant
  as well.  This gives enough \(S\)\nb-invariant
  bounded \(\dvr\)\nb-submodules
  in~\(\comb{A}\).
  So the induced action on~\(\comb{A}\) is uniformly bounded.

  If the action~\(\alpha\)
  on~\(A\)
  is uniformly bounded, then so is the action
  \(\mathrm{id}_B\otimes\alpha\)
  on \(B\otimes A\)
  for any bornological algebra~\(B\).
  In particular, the induced action on \(\dvf\otimes A\)
  is uniformly bounded.  Since the canonical map
  \(A\to \dvf\otimes A\)
  is \(S\)\nb-equivariant,
  the image~\(\torf{A}\)
  of~\(A\)
  in \(\dvf\otimes A\)
  is \(S\)\nb-invariant.
  The restriction of the uniformly bounded action of~\(S\)
  on \(\dvf\otimes A\)
  to this invariant subalgebra inherits uniform boundedness.  So the
  induced action on~\(\torf{A}\) is uniformly bounded.

  Any subset of linear growth in~\(A\)
  is contained in \(\sum_{j=0}^\infty \dvgen^j T^{j+1}\)
  for a bounded \(\dvr\)\nb-submodule~\(T\).
  Since~\(\alpha\)
  is uniformly bounded, \(T\)
  is contained in an \(\alpha\)\nb-invariant
  bounded \(\dvr\)\nb-submodule~\(U\).
  Then
  \(\sum_{j=0}^\infty \dvgen^j U^{j+1} \supseteq \sum_{j=0}^\infty
  \dvgen^j T^{j+1}\)
  is \(\alpha\)\nb-invariant
  and has linear growth.  So~\(\alpha\)
  remains uniformly bounded for the linear growth bornology.

  The uniform boundedness of the induced action on the dagger
  completion~\(A^\updagger\)
  follows from the inheritance properties above and
  Theorem~\ref{dagger_completion}.
\end{proof}

\begin{example}
  \label{exa:finite_S_uniformly_bounded}
  Let~\(S\)
  be a finite monoid.  Any bounded action of~\(S\)
  by bornological algebra endomorphisms is uniformly bounded because
  we may take \(T = \sum_{s\in S} \alpha_s(U)\)
  in Definition~ \ref{def:unif_bounded}.
\end{example}

\begin{example}
  \label{affine_action}
  We describe a uniformly bounded action of~\(\Z\)
  on the polynomial algebra \(A\defeq \dvr[x_1,\dotsc,x_n]\)
  with the fine bornology.  So a subset of~\(A\)
  is bounded if and only if it is contained in
  \((\dvr + \dvr x_1 + \dotsb + \dvr x_n)^k\)
  for some \(k \in \N_{\ge 1}\).
  Let \(a\in \mathrm{GL}_n(\dvr) \subseteq \Endo(\dvr^n)\)
  and \(b \in \dvr^n\).  Then
  \[
  \alpha_1\colon \dvr[x_1,\dotsc, x_n] \to
  \dvr[x_1,\dotsc,x_n],\qquad
  (\alpha_1 f)(x) \defeq f(a x + b),
  \]
  is an algebra automorphism~\(\alpha_1\)
  of~\(A\)
  with inverse \((\alpha_1^{-1} f)(x) \defeq f(a^{-1}(x - b))\).
  This generates an action of the group~\(\Z\)
  by \(\alpha_n \defeq \alpha_1^n\)
  for \(n\in\Z\).
  If a polynomial~\(f\)
  has degree at most~\(m\),
  then the same is true for \(\alpha_1 f\)
  and \(\alpha_{-1} f\),
  and hence for \(\alpha_n f\)
  for all \(n\in\Z\).
  That is, the \(\dvr\)\nb-submodules
  \((\dvr + \dvr x_1 + \dotsb + \dvr x_n)^k\)
  in~\(A\)
  for \(k\in\N\)
  are \(\alpha\)\nb-invariant.
  So the action~\(\alpha\)
  on~\(A\)
  is uniformly bounded.
  Proposition~\ref{pro:uniformly_bounded_induced_actions} implies that
  the induced action on \(\dvr[x_1,\dotsc,x_n]^\updagger\)
  is uniformly bounded as well.
\end{example}

\begin{proposition}
  \label{pro:dagger_completion_crossed_uniformly_bounded}
  Let~\(S\)
  be a finitely generated monoid with word length function~\(\ell\).
  Let~\(A\)
  be a dagger algebra and let \(\alpha\colon S\to\Endo(A)\)
  be a uniformly bounded action by algebra endomorphisms.  Then
  \((A\rtimes S)^\updagger \subseteq \prod_{s\in S} A\).
  A formal series \(\sum_{s\in S} a_s \delta_s\)
  with \(a_s\in A\)
  for all \(s\in S\)
  belongs to \((A\rtimes S)^\updagger\)
  if and only if there are \(\varepsilon>0\)
  and \(T\in\bdd_A\)
  with \(a_s \in \dvgen^{\floor{\varepsilon \ell(s)}} T\)
  for all \(s\in S\),
  and a set of formal series is bounded in \((A\rtimes S)^\updagger\)
  if and only if \(\varepsilon>0\)
  and \(T\in\bdd_A\) for its elements may be chosen uniformly.
\end{proposition}

\begin{proof}
  We first describe the linear growth bornology on~\(A\rtimes S\).
  Let~\(\bdd'\)
  be the set of all subsets \(U\subseteq A\rtimes S\)
  for which there are \(T\in\bdd_A\)
  and \(\varepsilon>0\)
  such that any element of~\(U\)
  is of the form \(\sum_{s\in S} a_s \delta_s\)
  with \(a_s \in \dvgen^{\floor{\varepsilon \ell(s)}} T\)
  for all \(s\in S\).
  We claim that~\(\bdd'\)
  is the linear growth bornology on \(A\rtimes S\).
  The inclusion \(\dvr[S] \subseteq A\rtimes S\)
  induces a bounded algebra homomorphism
  \(\ling{\dvr[S]} \to \ling{(A\rtimes S)}\).
  We have already described the linear growth bornology on~\(\dvr[S]\)
  in Section~\ref{sec:dagger_monoid}.  This implies easily that all
  subsets in~\(\bdd'\)
  have linear growth: write
  \(\dvgen^{\floor{\varepsilon \ell(s)}} a'_s \delta_s = a'_s\cdot
  \dvgen^{\floor{\varepsilon \ell(s)}} \delta_s\).
  We claim, conversely, that any subset of \(A\rtimes S\)
  of linear growth is contained in~\(\bdd'\).
  All bounded subsets of~\(A\rtimes S\)
  are contained in~\(\bdd'\).
  It is routine to show that~\(\bdd'\)
  is a \(\dvr\)\nb-algebra
  bornology.  We only prove that the bornology~\(\bdd'\)
  has linear growth.  Since~\(\alpha\)
  is uniformly bounded, any \(T\in\bdd_A\)
  is contained in a bounded, \(\alpha\)\nb-invariant
  \(\dvr\)\nb-submodule~\(T_2\).
  Then \(T_3 \defeq \sum_{j=0}^\infty \dvgen^j T_2^{j+1}\)
  is a bounded, \(\alpha\)\nb-invariant
  \(\dvr\)\nb-submodule
  with \(\dvgen \cdot T_3^2 \subseteq T_3\)
  and \(T \subseteq T_3\)
  (see
  \cite{Cortinas-Cuntz-Meyer-Tamme:Nonarchimedean}*{Equation~(5)}).  If
  \(a_s \in \dvgen^{\floor{\varepsilon \ell(s)}} T_3\)
  and \(a_t \in \dvgen^{\floor{\varepsilon \ell(t)}} T_3\), then
  \[
  \dvgen^2\cdot a_s\cdot \alpha_t
  \in \dvgen^{2+\floor{\varepsilon \ell(s)} +\floor{\varepsilon \ell(t)}} T_3^2
  \subseteq \dvgen^{\floor{\varepsilon \ell(s\cdot t)}} \dvgen T_3^2
  \subseteq \dvgen^{\floor{\varepsilon \ell(s\cdot t)}} T_3
  \]
  because
  \(1+\floor{\varepsilon \ell(s)} +\floor{\varepsilon \ell(t)} \ge
  \floor{\varepsilon \ell(s\cdot t)}\).  This implies
  \[
  \dvgen^2 \cdot \sum_{s\in S} \dvgen^{\floor{\varepsilon \ell(s)}} T_3 \delta_s *
  \sum_{t\in S} \dvgen^{\floor{\varepsilon \ell(t)}} T_3 \delta_t
  \subseteq
  \sum_{s,t\in S} \dvgen^{\floor{\varepsilon \ell(s\cdot t)}} T_3 \delta_{s t}
  \subseteq
  \sum_{s\in S} \dvgen^{\floor{\varepsilon \ell(s)}} T_3 \delta_s.
  \]
  So any subset in~\(\bdd'\)
  is contained in \(U\in \bdd'\)
  with \(\dvgen^2 \cdot U^2 \subseteq U\).
  By induction, this implies \((\dvgen^2 U)^k \cdot U \subseteq U\)
  for all \(k\in\N\).
  Hence \(\sum_{j=0}^\infty \dvgen^{2 k} U^{k+1}\)
  is in~\(\bdd'\).
  Now Lemma~\ref{equivalent_semidagger} shows that the
  bornology~\(\bdd'\)
  is semi-dagger.  This proves the claim that~\(\bdd'\)
  is the linear growth bornology on~\(A\rtimes S\).

  Since~\(A\)
  as a dagger algebra is bornologically torsion-free, so is
  \(A\rtimes S\).
  So \((A\rtimes S)^\updagger\)
  is the completion of \(\ling{(A\rtimes S)} = (A\rtimes S,\bdd')\).
  It is routine to identify this completion with the bornological
  \(\dvr\)\nb-module
  described in the statement.
\end{proof}

Propositions \ref{pro:uniformly_bounded_induced_actions}
and~\ref{pro:dagger_completion_crossed_uniformly_bounded} describe the
dagger completion of \(A\rtimes S\)
for a uniformly bounded action of~\(S\)
on~\(A\)
even if~\(A\)
is not a dagger algebra.  Namely, the universal properties of the
crossed product and the dagger completion imply
\begin{equation}
  \label{eq:dagger_complete_twice}
  (A\rtimes S)^\updagger \cong (A^\updagger \rtimes S)^\updagger.
\end{equation}

\begin{example}
  \label{exa:dagger_crossed_polynomial}
  Let~\(\alpha\)
  be the uniformly bounded action of~\(\Z\)
  on \(\dvr[x_1,\dotsc,x_k]\)
  from Example~\ref{affine_action}.  The induced
  action~\(\alpha^\updagger\)
  on \(\dvr[x_1,\dotsc,x_k]^\updagger\)
  is also uniformly bounded by
  Proposition~\ref{pro:uniformly_bounded_induced_actions}.
  And~\eqref{eq:dagger_complete_twice} implies
  \[
  (\dvr[x_1,\dotsc,x_k] \rtimes_\alpha \Z)^\updagger
  \cong (\dvr[x_1,\dotsc,x_k]^\updagger \rtimes_{\alpha^\updagger} \Z)^\updagger.
  \]
  The latter is described in
  Proposition~\ref{pro:dagger_completion_crossed_uniformly_bounded}.
  Namely,
  \((\dvr[x_1,\dotsc,x_k]^\updagger \rtimes_{\alpha^\updagger}
  \Z)^\updagger\)
  consists of those formal series \(\sum_{n\in \Z} a_n \delta_n\)
  with \(a_n \in \dvr[x_1,\dotsc,x_k]^\updagger\)
  for which there are \(\varepsilon>0\)
  and a bounded \(\dvr\)\nb-submodule~\(T\)
  in \(\dvr[x_1,\dotsc, x_k]^\updagger\)
  such that \(a_n \in \dvgen^{\floor{\varepsilon \abs{n}}} T\)
  for all \(n\in \Z\);
  notice that~\(\abs{n}\)
  is indeed a length function on~\(\Z\).
  And a subset is bounded if some pair \(\varepsilon,T\)
  works for all its elements.

  We combine this with the description of bounded subsets of
  \(\dvr[x_1,\dotsc, x_k]^\updagger\)
  in Section~\ref{sec:dagger_monoid}: there is some \(\delta>0\)
  so that a formal power series \(\sum_{m\in \N^k} b_m x^m\)
  belongs to~\(T\)
  if and only if \(b_m \in \dvgen^{\floor{\delta \abs{m}}} \dvr\)
  for all \(m\in\N^k\).
  Here we use the length function
  \(\abs{(m_1,\dotsc,m_k)} = \sum_{j=1}^k m_j\).
  We may merge the parameters \(\varepsilon,\delta>0\)
  above, taking their minimum.  So
  \((\dvr[x_1,\dotsc,x_k]\rtimes \Z)^\updagger\)
  consists of the formal series
  \(\sum_{n\in\Z,m\in\N^k} a_{n,m} x^m \delta_n\)
  with
  \(a_{n,m} \in\dvgen^{\floor{\varepsilon (\abs{n} + \abs{m})}} \dvr\)
  or, equivalently,
  \(\nu(a_{n,m}) +1 > \varepsilon (\abs{n} + \abs{m})\)
  for all \(n\in\Z\), \(m\in\N^k\).
\end{example}

If the action of~\(S\)
on~\(A\)
is not uniformly bounded, then the linear growth bornology
on~\(A\rtimes S\)
becomes much more complicated.  It seems unclear whether the
description below helps much in practice.  Let \(F\subseteq S\)
be a finite generating subset containing~\(1\).
Any bounded subset of \(A\rtimes S\)
is contained in \(\bigl(\sum_{s\in F} T\cdot\delta_s\bigr)^N\)
for some \(N\in\N\)
and some \(T\in\bdd_A\)
with \(1\in T\).
Therefore, a subset of~\(A\rtimes S\)
has linear growth if and only if it is contained in the
\(\dvr\)\nb-submodule generated by
\[
\bigcup_{n=1}^\infty \dvgen^{\floor{\varepsilon n}}(T\cdot \setgiven{\delta_s}{s\in F})^n
\]
for some \(\varepsilon>0\),
\(T\in\bdd_A\).
Using the definition of the convolution, we may rewrite the latter set
as
\[
\bigcup_{n=1}^\infty \bigcup_{s_1,\dotsc,s_n\in F} \dvgen^{\floor{\varepsilon n}} \cdot
T\cdot \alpha_{s_1}(T)\cdot \alpha_{s_1 s_2}(T) \dotsm \alpha_{s_1 \dotsm s_{n-1}}(T) \,
\delta_{s_1\dotsm s_n}.
\]
The resulting \(\dvr\)\nb-module
is the sum \(\sum_{s\in S} U_s \delta_s\),
where~\(U_s\)
is the \(\dvr\)\nb-submodule
of~\(A\)
generated by finite products
\[
\setgiven{\dvgen^{\floor{\varepsilon n}}
\cdot T\cdot \alpha_{s_1}(T) \dotsm \alpha_{s_1\dotsm s_{n-1}}(T)}
{n\in\N_{\ge1},\ s_1,\dotsc,s_n\in F,\ s_1\dotsm s_n = s}.
\]
Here taking a factor \(1 \in T\)
is allowed.  Thus we may leave out a factor
\(\alpha_{s_1\dotsm s_i}(T)\).
This has the same effect as increasing~\(n\)
by~\(1\)
and putting \(s_i = s_i^1\cdot s_i^2\)
with \(s_i^1,s_i^2\in F\).
Since~\(F\)
generates~\(S\)
as a monoid, we may allow arbitrary \(s_i\in S\)
when we change the exponent of~\(\dvgen\)
appropriately.  Namely, we must then replace~\(n\)
in the exponent of~\(\dvgen\)
by the number of factors in~\(F\)
that are needed to produce the desired elements~\(s_i\),
which is \(\ell_{\ge1}(s_1) + \dotsb + \ell_{\ge1}(s_n)\),
where \(\ell_{\ge1}(1) = 1\)
and \(\ell_{\ge1}(s) = \ell(s)\)
for \(s\in S\setminus\{1\}\).
As a result, \(U_s\)
is the \(\dvr\)\nb-submodule of~\(A\) generated by
\begin{multline*}
  \dvgen^{\floor{\varepsilon (\ell_{\ge1}(s_1) + \dotsb +
      \ell_{\ge1}(s_n))}} \cdot x_0\cdot \alpha_{s_1}(x_1) \dotsm
  \alpha_{s_1\dotsm s_{n-1}}(x_{n-1}),\\
  n\in\N_{\ge1},\ x_0,\dotsc,x_{n-1}\in T,\ s_1,\dotsc,s_n\in S,\ s_1\dotsm s_n = s.
\end{multline*}

Now assume that~\(S\)
is a group, not just a monoid.  Then any sequence of elements
\(g_1,\dotsc,g_n\in S\)
may be written as \(g_i = s_1\dotsm s_i\)
by putting \(s_i\defeq g_{i-1}^{-1} g_i\)
with \(g_0\defeq 1\).
So~\(U_g\) is the \(\dvr\)\nb-submodule of~\(A\) generated by
\begin{multline*}
  \dvgen^{\floor{\varepsilon (\ell_{\ge1}(g_0^{-1} g_1) + \ell_{\ge1}(g_1^{-1} g_2)
      + \dotsb + \ell_{\ge1}(g_{n-1}^{-1} g_n))}}
  \cdot \alpha_{g_0}(x_0)\cdot \alpha_{g_1}(x_1) \dotsm \alpha_{g_{n-1}}(x_{n-1}),\\
  n\in\N_{\ge1},\ x_0,\dotsc,x_{n-1} \in T,\ g_0,\dotsc,g_n\in S,\ g_0=1,\ g_n=g.
\end{multline*}
These subsets~\(U_g\)
for fixed \(T\)
and~\(\varepsilon\)
depend on~\(g\)
in a complicated way.  The bornology on~\(A\rtimes G\)
generated by these subsets is, however, also generated by the sets of
the form
\(\sum_{g\in G} \dvgen^{\varepsilon \ell(g)}\cdot U\,\delta_g\),
where~\(U\) is the \(\dvr\)\nb-submodule of~\(A\) generated by
\begin{multline*}
  \dvgen^{\floor{\varepsilon (\ell_{\ge1}(g_0^{-1} g_1) + \ell_{\ge1}(g_1^{-1} g_2)
      + \dotsb + \ell_{\ge1}(g_{n-2}^{-1} g_{n-1}))}}
  \cdot \alpha_{g_0}(x_0)\cdot \alpha_{g_1}(x_1) \dotsm \alpha_{g_{n-1}}(x_{n-1}),\\
  n\in\N_{\ge1},\ x_0,\dotsc,x_{n-1} \in T,\ g_0,\dotsc,g_{n-1}\in S,\ g_0=1,
\end{multline*}
for some \(T\in\bdd_A\), \(\varepsilon>0\).
The reason is that
\[
\ell(g_n) - \ell(g_{n-1}) \le \ell(g_{n-1}^{-1} g_n) \le
\ell(g_{n-1}) + \ell(g_n)
\]
and \(\ell(g_{n-1}) \le \sum_{j=1}^{n-1} \ell(g_{j-1}^{-1} g_j)\).
Therefore, replacing the exponents of~\(\dvgen\)
as above does not change the bornology on \(A\rtimes G\)
that is generated by the sets above when \(\varepsilon>0\) varies.

\end{document}